\pdfoutput=1
\documentclass{scrartcl}

\usepackage{paralist}
\usepackage[intlimits]{amsmath}
\usepackage{amsthm,amssymb,mathtools,bm,ifpdf,relsize,authblk}
\usepackage{url}
\usepackage[abbrev]{amsrefs}
\usepackage[usenames,dvipsnames]{color}

\usepackage[T1]{fontenc}
\usepackage{fix-cm}
\usepackage{microtype}
\usepackage{tocstyle}
\usepackage{hyperref}

\newtheorem{thm}{Theorem}[section]

\newtheorem{lemma}[thm]{Lemma}
\newtheorem*{lemma*}{Lemma}
\newtheorem{prop}[thm]{Proposition}
\newtheorem{cor}[thm]{Corollary}
\newtheorem*{cor*}{Corollary}

\theoremstyle{definition}

\newtheorem{ex}[thm]{Example}
\newtheorem*{ex*}{Example}

\newtheorem*{exs*}{Example}

\newtheorem{defn}[thm]{Definition}

\newtheorem*{defn*}{Definition}
\newtheorem*{defns*}{Definition}

\newtheorem{question}[thm]{Question}

\newtheorem{notation}[thm]{Notation}
\newtheorem*{notation*}{Notation}

\newtheorem{rem}[thm]{Remark}
\newtheorem*{rem*}{Remark}

\numberwithin{equation}{section}
\allowdisplaybreaks[1]

\renewcommand{\le}{\leqslant}
\renewcommand{\ge}{\geqslant}

\def\emptyset{\varnothing}

\def\emph{}
\DeclareTextFontCommand{\bfemph}{\bf}
\DeclareTextFontCommand{\itemph}{\it}
\def\emph{\bfemph}

\makeatletter
\def\blankfootnote{\xdef\@thefnmark{}\@footnotetext}

\newcommand*{\textlabel}[2]{%
  \edef\@currentlabel{#1}
  \phantomsection
  #1\label{#2}
}
\makeatother

\newcommand{\idx}[1]{\lvert#1\rvert}
\newcommand{\divides}[2]{\ensuremath{ {#1} \mid {#2} }}
\newcommand{\ndivides}[2]{\ensuremath{ {#1} \nmid {#2} }}

\newcommand{\onto}{\twoheadrightarrow}
\newcommand{\into}{\rightarrowtail}
\newcommand{\incl}{\hookrightarrow}

\newcommand{\FF}{\mathbf{F}}

\newcommand{\GG}{\ensuremath{\mathbf{G}}}
\newcommand{\G}{\ensuremath{\mathsf{G}}}
\newcommand{\HH}{\ensuremath{\mathbf{H}}}
\let\H\undefined
\newcommand{\H}{\ensuremath{\mathsf{H}}}

\newcommand{\QQ}{\mathbf{Q}}
\newcommand{\NN}{\mathbf{N}}
\newcommand{\ZZ}{\mathbf{Z}}
\newcommand{\CC}{\mathbf{C}}
\newcommand{\MM}{\mathbf M}

\newcommand{\RR}{\mathbf{R}}
\newcommand{\fg}{\ensuremath{\mathfrak g}}

\newcommand{\Places}{\ensuremath{\mathcal V}}
\newcommand{\RF}{\ensuremath{\mathfrak K}}

\newcommand{\XX}{\ensuremath{\bm X}}
\newcommand{\YY}{\ensuremath{\bm Y}}
\newcommand{\Zeta}{\ensuremath{\mathsf{Z}}}

\newcommand{\ee}{\ensuremath{\bm e}}
\newcommand{\ess}{\ensuremath{\bm s}}

\newcommand{\uu}{\ensuremath{\bm u}}

\newcommand{\xx}{\ensuremath{\bm x}}
\newcommand{\yy}{\ensuremath{\bm y}}

\newcommand{\fp}{\mathfrak{p}}
\newcommand{\fP}{\mathfrak{P}}
\newcommand{\fo}{\mathfrak{o}}
\newcommand{\fO}{\mathfrak{O}}

\newcommand{\cC}{\mathcal{C}}
\newcommand{\cD}{\mathcal{D}}
\newcommand{\cH}{\mathcal{H}}
\newcommand{\cN}{\mathcal{N}}
\newcommand{\cO}{\mathcal{O}}
\newcommand{\cP}{\mathcal{P}}
\newcommand{\cQ}{\mathcal{Q}}
\newcommand{\cR}{\mathcal{R}}
\newcommand{\cS}{\mathcal{S}}
\newcommand{\cT}{\mathcal{T}}
\newcommand{\cZ}{\mathcal{Z}}

\DeclareMathOperator{\GL}{GL}
\DeclareMathOperator{\Mat}{M}
\DeclareMathOperator{\Spec}{Spec}
\DeclareMathOperator{\topo}{top}

\newcommand{\Torus}{\mathbf{T}}
\newcommand{\Euler}{\ensuremath{\chi}}
\newcommand{\Orth}{\RR_{\ge 0}}
\newcommand{\StrictOrth}{\RR_{> 0}}

\DeclareMathOperator{\dd}{d\!}

\newcommand{\dtimes}{\ensuremath{\,\cdotp}}
\newcommand{\blank}{\ensuremath{{-}}}
\newcommand{\card}[1]{\lvert#1\rvert}
\newcommand{\noof}[1]{\#{#1}}
\newcommand{\bil}[2]{\ensuremath{\langle {#1}, {#2} \rangle }}

\DeclarePairedDelimiter{\abs}{\lvert}{\rvert}
\DeclarePairedDelimiter{\norm}{\lVert}{\rVert}
\DeclarePairedDelimiter{\red}{\lfloor}{\rfloor}

\DeclareMathOperator{\supp}{supp}
\DeclareMathOperator{\init}{in}

\DeclareMathOperator{\Newton}{New}
\DeclareMathOperator{\NormalCone}{N}
\DeclareMathOperator{\Real}{Re}

\DeclareMathOperator{\conv}{conv}
\DeclareMathOperator{\Irr}{Irr}
\DeclareMathOperator{\irr}{irr}

\newcommand{\llb}{\ensuremath{[\![ }}
\newcommand{\llp}{\ensuremath{(\!( }}

\newcommand{\rrb}{\ensuremath{]\!] }}
\newcommand{\rrp}{\ensuremath{)\!) }}

\ifpdf
  \hypersetup{pdftitle={Topological representation zeta functions of unipotent groups},
    pdfauthor={Tobias Rossmann}
  }
\fi
\title{Topological representation zeta functions of unipotent groups}
\author{Tobias Rossmann}
\affil{\small Fakult\"at f\"ur Mathematik, Universit\"at Bielefeld, D-33501
  Bielefeld, Germany}
\date{March 2015}
\begin{document}

\maketitle
\thispagestyle{empty}

\begin{abstract}
  \small
  Inspired by work surrounding Igusa's local zeta function,
  we introduce topological representation zeta functions of unipotent algebraic
  groups over number fields.
  These group-theoretic invariants capture common features of
  established $p$-adic representation zeta functions associated with pro-$p$
  groups derived from unipotent groups.
  We investigate fundamental properties of the topological
  zeta functions considered here.
  We also develop a method for computing them under non-degeneracy assumptions.
  As an application, among other things, we obtain a complete
  classification of  topological representation zeta functions of unipotent
  algebraic groups of dimension at most six.
\end{abstract}

\blankfootnote{\indent{\itshape 2010 Mathematics Subject Classification.}
  11M41, 
  20F69, 
  20G30, 
  20F18, 
  20C15, 
  14M25. 
  

  This work is supported by the DFG Priority Programme
  ``Algorithmic and Experimental Methods in Algebra, Geometry and Number
  Theory'' (SPP 1489).}

\section{Introduction}
\label{s:intro}

\paragraph{Enumerating representations.}
Representation zeta functions of groups are Dirichlet series enumerating
irreducible finite-dimensional complex representations up to
suitable notions of equivalence. 
For instance, given a group $G$, let $r_n(G)$ be the possibly
infinite number of irreducible representations $G \to
\GL_n(\CC)$ counted up to equivalence of representations in the usual sense;
if $G$ is a topological group, then we only consider continuous representations.
For various interesting classes of groups, it turns out that the numbers
$r_n(G)$ are polynomially bounded as a function of $n$. 
For such groups, we may then consider the representation zeta function $\sum_{n=1}^\infty
r_n(G)n^{-s}$, where $s$ is a complex variable.
A part of the theory of representation growth, the study of such zeta
functions is an active area, see \cites{Klo13,Vol14} for surveys.
Among the classes of groups of major interest are
arithmetic groups associated with semisimple algebraic groups \cite{LM04,LL08},
compact $p$-adic analytic groups \cite{JZ06, AKOV13}, as well as nilpotent and
unipotent groups \cite{HMRC14,SV14}.

\paragraph{Representation zeta functions of nilpotent and unipotent groups:\! twisting.}
For a finitely generated infinite nilpotent group $G$,
the number~$r_1(G)$ is infinite.
Hrushovski and Martin \cite{HMRC14} (first version, 2006)
studied ``twist-isoclasses'' of representations.
Here, two representations $\varrho_1,\varrho_2\colon G\to \GL_n(\CC)$ of
an arbitrary group $G$ are \emph{twist-equivalent} if $\varrho_1$ is equivalent
to $\alpha \otimes_{\CC}  \varrho_2$ in the usual sense for a 
$1$-dimensional representation $\alpha\colon G\to \CC^\times$;
for a topological group $G$, we require $\varrho_1$, $\varrho_2$, and $\alpha$
to be continuous.
Twist-equivalence classes of representations are called \emph{twist-isoclasses}.
Let $\tilde r_n(G)$ denote the number of twist-isoclasses of (continuous)
irreducible representations $G \to \GL_n(\CC)$. 
If $G$ is finitely generated nilpotent or the pro-$p$ completion
of such a group, then the $\tilde r_n(G)$ are
polynomially bounded and
the \emph{representation zeta function} of $G$ is 
\[
\zeta_G(s) := \sum_{n=1}^\infty \tilde r_n(G) n^{-s}.
\]
Such zeta functions were studied in~\cite[\S 3.4]{Vol10} and \cite[\S 8]{HMRC14}.
Denoting the pro-$p$ completion of a finitely generated nilpotent group $G$ by
$\hat G_p$, a central theme in both cited sources is the behaviour
of $\zeta_{\hat G_p}(s)$ under variation of $p$.
Extending \cite[\S 3.4]{Vol10},
Stasinski and Voll~\cite{SV14} studied representation zeta functions enumerating
twist-isoclasses of groups arising from unipotent group schemes.
This is essentially the point of view taken here.

\paragraph{Previously considered types of topological zeta functions.}
Classical topological zeta functions are singularity invariants of hypersurfaces.
Given $f\in \ZZ[X_1,\dotsc,X_n]$,
the topological zeta function of $f$ was
first defined by Denef and Loeser~\cite{DL92} by means of a limit ``$p\to 1$''
applied to Igusa's $p$-adic zeta function associated with $f$.
Later, the topological zeta function of $f$ was reinterpreted 
within the framework of motivic integration \cite{DL98}.
Not only do topological zeta functions retain crucial features of their
$p$-adic ancestors (see \cite[Thm~2.2]{DL92}), they have also been found to be
more amenable to both theoretical investigations and explicit computations, see
e.g.~\cites{LVP11,NV12} for some recent developments.

Using connections between subgroup and subring zeta functions and
$p$-adic integration going back to \cite{GSS88},
topological subalgebra and subgroup zeta functions of Lie algebras and nilpotent
groups, respectively, were introduced  by du~Sautoy and
Loeser \cite{dSL04}.
In~\cites{topzeta,topzeta2}, the author developed systematic means of computing
these zeta functions in favourable situations.
Substantial evidence indicates that they possess remarkable properties 
which are not present in the hypersurface case, see \cite[\S 8]{topzeta}.

\paragraph{Topological representation zeta functions of unipotent groups.}
The purpose of the present article is to initiate the study of topological
representation zeta functions and to establish them as interesting
group-theoretic invariants that deserve further attention.

First, after recalling some key results from \cite{SV14}
in~\S\ref{s:local},  
given a unipotent algebraic group $\GG$ over a number field $k$,
we define the topological representation zeta function $\zeta_{\GG,\topo}(s) \in
\QQ(s)$ of $\GG$ in \S\ref{s:topological} using $p$-adic formulae from
\cites{Vol10,SV14} and ideas from \cite{DL92}.
For a brief and informal sketch, suppose that $k = \QQ$ and
choose an arbitrary $\ZZ$-form $\G$ of $\GG$ as an affine group scheme.
Then for sufficiently large primes $p$,
each $\G(\ZZ_p)$ is a pro-$p$ group whose representation zeta function
$\zeta_{\G(\ZZ_p)}(s)$ enumerating twist-isoclasses is defined as above,
where $\ZZ_p$ denotes the $p$-adic integers.
Informally, we then define $\zeta_{\GG,\topo}(s)$ to be the constant term of 
$\zeta_{\G(\ZZ_p)}(s)$ as a series in $p-1$.
As an example, a $\ZZ$-form of the Heisenberg group, $\HH$ say, over $\QQ$ is
given by
\[\H(R) := \begin{bmatrix}
  1 & R & R \\
  0  & 1 & R \\
  0  & 0  & 1
\end{bmatrix}
\le\GL_3(R)\]
for commutative rings $R$.
It is well-known that
$$\zeta_{\H(\ZZ_p)}(s) = \frac{1-p^{-s}}{1-p^{1-s}}$$ for all primes $p$.
By formally expanding $p^z = (1+(p-1))^z$ in $p-1$ using the binomial series, we 
obtain $\zeta_{\H(\ZZ_p)}(s) = \frac{s}{s-1} + \cO(p-1)$ whence
$\zeta_{\HH,\topo}(s) = \frac{s}{s-1}$.

Having rigorously defined topological representation zeta functions, 
we establish some of their basic properties in~\S\ref{s:properties}.
Perhaps most interestingly, we find that they always have degree zero in $s$.
In contrast, the degrees of topological zeta functions of polynomials can vary
wildly, and the degrees of topological subalgebra zeta functions are only
understood conjecturally, see \cite[Conj.~I]{topzeta}.
Along the way, we also briefly consider topological representation zeta
functions attached to perfect Lie algebras in the spirit of \cite{AKOV13}.

Computations (largely ad~hoc) of $p$-adic representation zeta functions
of various nilpotent groups can be found in the theses of Ezzat~\cite{Ezzat} and
Snocken~\cite{Snocken}.
Being derived from $p$-adic zeta functions, we expect topological representation
zeta functions to be more easily computable than the former.
Indeed, building on the author's previous work \cites{topzeta,topzeta2},
in \S\ref{s:compute}, we develop a method for directly computing topological
representation zeta functions at least under additional hypotheses. 
With this method at our disposal, in~\S\ref{s:examples}, we provide numerous
examples of topological representation zeta functions, including a complete
classification in dimension six (see Table~\ref{tab:six}, p.~\pageref{tab:six}).
As a result of our work, our knowledge of specific examples of topological
representation zeta functions associated with unipotent groups far exceeds the
$p$-adic case.
Finally, in \S\ref{s:questions}, we discuss open questions that arose from 
remarkable patterns exhibited by the known examples of topological
representation zeta functions.

\subsection*{\it Acknowledgement}
The author is indebted to Christopher Voll for many interesting conversations.

\section{Local representation zeta functions}
\label{s:local}

We explain how unipotent algebraic groups give rise to pro-$p$
groups.  
We then recall a result from \cite{Vol10,SV14} on the shapes of the
representation zeta functions of these groups.

\subsection{Basic facts on unipotent groups over number fields}
\label{ss:unipotent}

\paragraph{Unipotent algebraic groups.}
The following is folklore, see \cite[Ch.~IV, \S 2]{DG70} and \cite[Ch.~XV]{milneAGS}.
Let $F$ be a field. 
By an \emph{$F$-group}, we mean an algebraic group over $F$.
As a non-intrinsic characterisation of unipotence,
we say that an $F$-group is \emph{unipotent} if for some $n\ge 1$,
it is isomorphic to an algebraic subgroup of the $F$-group 
of upper unitriangular $n\times n$-matrices for some $n$.
If $F$ is perfect, then a connected linear $F$-group is unipotent
if and only if it admits a central series of closed algebraic subgroups whose
non-trivial factors are $F$-isomorphic to the additive group.

Let $F$ have characteristic zero.
Given a finite-dimensional nilpotent Lie $F$-algebra $\bm\fg$,
the Baker-Campbell-Hausdorff series associated with $\fg$ is a polynomial.
Using it, we may endow $\GG(A) := \bm\fg \otimes_F A$ with a group structure for
each commutative $F$-algebra $A$.
The functor $A\mapsto \GG(A)$ is represented by a unipotent $F$-group
which we again denote by $\GG$.
The construction of $\GG$ from $\bm \fg$ is functorial and $\bm\fg \mapsto \GG$
furnishes an equivalence between the categories of finite-dimensional nilpotent
Lie $F$-algebras and unipotent $F$-groups; a quasi-inverse
is given by the usual functor taking an algebraic group to its Lie algebra.

\paragraph{Notation for number fields and their places.}
Throughout this article, we let $k$ denote a number field with ring of integers $\fo$.
We let $\Places_k$ denote the set of all non-Archimedean places of $k$.
Given $v \in \Places_k$, we let $k_v$ denote the $v$-adic completion of $k$.
We further let $\fo_v$ and $\RF_v$ denote the valuation ring and residue field
of $k_v$, respectively, and write $q_v = \card{\RF_v}$.
Let $\fp_v \in \Spec(\fo)$ be the prime ideal corresponding to $v$ and
let $p_v$ be the rational prime contained in $\fp_v$.
For a $p$-adic field $K$,
i.e.~a finite extension of the field $\QQ_p$ of $p$-adic numbers,
let $\fO_K$ denote the valuation ring of $K$,
let $\fP_K$ be the maximal ideal of $\fO_K$, and write $q_K =
\card{\fO_K/\fP_K}$. 
We let $\abs{\dtimes}_K$ and $\norm{\dtimes}_K$ denote the usual $\fP_K$-adic
absolute value and maximum norm, respectively; in particular, $\abs{\pi}_K =
q_K^{-1}$ for $\pi\in \fP_K^{\phantom 1}\!\setminus\fP_K^2$.

\paragraph{Integral forms and Lie groups.}
Let $\GG$ be a linear algebraic group over $k$.
By an \emph{$\fo$-form} of $\GG$, we mean an affine 
group scheme $\G$ of finite type over $\fo$ with
$\G \otimes_{\fo} k \approx_k \GG$.
Write $\GL_{n,R} = \GL_n\otimes_{\ZZ} R$.
Every $k$-isomorphism from $\GG$ onto an algebraic subgroup of some $\GL_{n,k}$
provides us with an $\fo$-form of $\GG$
as a closed subgroup scheme of $\GL_{n,\fo}$, see~\cite[\S 1]{Fom97}.
If $\G$ is an arbitrary $\fo$-form of $\GG$, there exists
$N\in \NN$ such that $\G \otimes_\fo \fo[1/N]$ embeds as a closed subgroup scheme into
$\GL_{n,\fo[1/N]}$.
Define $S = \{ v\in \Places_k : \ndivides{p_v} N\}$.
If $v\in \Places_k\setminus S$, then we may regard each $\G(\fo_v) \le \GL_n(\fo_v)$
as a compact $p_v$-adic Lie group.
If $\GG$ is unipotent, then
for almost all $v\in \Places_k\setminus S$, the group $\G(\fo_v)$ is a
torsion-free, nilpotent, finitely generated pro-$p_v$ group. 
While the family $(\G(\fo_v))_{v\in \Places_k\setminus S}$ of topological groups
depends on the choices made, it does so in a mild way.
Namely, if $\tilde\G$ is another $\fo$-form of $\GG$, then there exists a
finite set $\tilde S \subset \Places_k$ such that
$\G(\fo_v) \approx \tilde\G(\fo_v)$ as topological groups
for $v\in \Places_k\setminus \tilde S$; 
we may even assume that $\G(\fO_K) \approx \tilde\G(\fO_K)$ 
for finite extensions $K/k_v$.

\subsection{Representation zeta functions associated with unipotent groups}
\label{ss:local}

We recall the setting and statement of \cite[Thm~A]{SV14} in a form convenient
for our purposes.
First, we record the following elementary fact.

\begin{lemma}[{Cf.\ \cite[Lemma~2.1]{SV14}}]
  \label{lem:ptrg}
  Let $G$ be a finitely generated nilpotent pro-$p$ group.
  Then $\tilde r_n(G) < \infty$ for all $n\in \NN$ and $\tilde r_n(G) =
  \cO(n^\alpha)$ for some real number $\alpha \ge 0$.
\end{lemma}
\begin{proof}
  Since $G$ is $p$-adic analytic, it has polynomial subgroup growth
  \cite[Thm~3.19]{DdSMS99}.
  The proof of \cite[Lemma~2.1]{SV14} for the case of $\mathcal T$-groups now
  carries over to the present case.
\end{proof}

In particular, the definition of $\zeta_G(s)$ given in the introduction makes
perfect sense if $G$ is any finitely generated nilpotent pro-$p$ group.
Note that since, as is well-known, finite $p$-groups are monomial, $\tilde
r_n(G) \not= 0$ only if $n$ is a power of $p$.
By \S\ref{ss:unipotent}, if $\G$ is an $\fo$-form of a unipotent 
$k$-group, then Lemma~\ref{lem:ptrg} applies to almost all of the
groups $\G(\fo_v)$ for $v\in \Places_k$.
The following result explains the behaviour of $\zeta_{\G(\fo_v)}(s)$
under variation of $v$.

\begin{thm}[{Pf of \cite[Thm~A]{SV14}; cf.~\cite[Thm~D]{Vol10}}]
  \label{thm:SV_ThmA}
  Let $\G$ be an $\fo$-form of a unipotent $k$-group.
  Then there are separated $\fo$-schemes $V_1,\dotsc,V_r$ of finite type,
  rational functions $W_1,\dotsc,W_r \in \QQ(X,Y)$, and a finite set
  $S \subset \Places_k$ such that if $v\in \Places_k\setminus S$ and $K$ is any
  finite extension of $k_v$,
  then $\zeta_{\G(\fO_K)}(s) = \sum\limits_{i=1}^r \noof{V_i(\fO_K/\fP_K)} \dtimes
  W_i(q_K^{\phantom{s_1}},q_K^{-s})$---in particular, $\zeta_{\G(\fO_K)}(s)$ is
  rational in $q_K^{-s}$ over $\QQ$ and admits meromorphic continuation to all of $\CC$.
\end{thm}

\begin{rem}
  \label{r:SV_form}
  Let $\bm\fg$ be a finite-dimensional nilpotent Lie $k$-algebra.
  Choose an $\fo$-form $\fg$ of $\bm\fg$ 
  which is free as an $\fo$-module and 
  which satisfies $[\fg,\fg]\subset c!\fg$, where $c$ is the nilpotency class of
  $\bm\fg$.
  Stasinski and Voll~\cite[\S 2.1.2]{SV14} used the
  following $\fo$-form $\G$ of the unipotent $k$-group corresponding to
  $\bm\fg$. 
  Namely, for a commutative $\fo$-algebra $R$, define a group structure  on
  $\G(R) =\fg\otimes_{\fo} R$ using the Baker-Campbell-Hausdorff series,
  exactly as in \S\ref{ss:unipotent}.
  The resulting group scheme $\G$ over $\fo$ is unipotent in the sense of
  \cite[\S 2.1.1]{SV14}, i.e.~$\G$ is affine, smooth,
  and its geometric fibres are connected unipotent algebraic groups.
  For our purposes, we may discard finitely many places of $k$ as needed
  whence the particular choice of an $\fo$-form is immaterial.
\end{rem}

\section{Topological representation zeta functions}
\label{s:topological}

In the preceding section,
given a unipotent algebraic group $\GG$ over $k$, after choosing an arbitrary
$\fo$-form $\G$ of $\GG$, we obtained a family $(\G(\fo_v))_{v\in
  \Places_k\setminus S}$ of groups and associated representation zeta functions.
In this section,
we define the topological representation zeta function of $\GG$
by means of a limit ``$q_v \to 1$'' applied to the zeta functions
$\zeta_{\G(\fo_v)}(s)$. 

\subsection{Taking the limit ``\texorpdfstring{$q \to 1$}{q -> 1}''}

In the study of local zeta functions, we encounter families of
functions of the form $W(q_v^{\phantom {s_1}}\!\!,q_v^{-s_1},\dotsc,q_v^{-s_l})$, where $W \in
\QQ(X,Y_1,\dotsc,Y_l)$ and $v$ runs over (almost all elements of) $\Places_k$.
In this subsection, which summarises \cite[\S 5.1]{topzeta},
we describe conditions on the shape of $W$ that allow us to
pass to the limit ``$q_v\to 1$'', denoted $\red W$ below, by 
taking the constant term of $W(q_v,q_v^{-s_1},\dotsc,q_v^{-s_l})$ symbolically
expanded as a series in $q_v-1$.

Given a polynomial $e \in \QQ[s_1,\dotsc,s_l]$, using the binomial series,
we formally expand 
\[
X^e := (1 + (X-1))^e := \sum_{d=0}^\infty \binom e d (X-1)^d
\in \QQ[s_1,\dotsc,s_l]\llb X-1\rrb.
\]
The rule $f \mapsto f(X,X^{-s_1},\dotsc,X^{-s_l})$
then extends to an embedding of $\QQ(X,Y_1,\dotsc,Y_l)$ into the field
$\QQ(s_1,\dotsc,s_l)\llp X-1\rrp$ of formal Laurent series in $X-1$
over $\QQ(s_1,\dotsc,s_l)$. 

\begin{defn}
  \label{d:MM}
  Let $\MM[X,Y_1,\dotsc,Y_l]$ be the $\QQ$-algebra
  consisting of all those rational functions $W \in \QQ(X,Y_1,\dotsc,Y_l)$ 
  satisfying the following two conditions:
\begin{enumerate}[(a)]
\item
  \label{d:MMa}
  $W$ can be written in the form
  \[
  W = f\dtimes\prod_{(a,b)\in \ZZ^{1+l}\setminus \{0\}}(1-X^a Y_1^{b_1}\dotsb
  Y_l^{b_l})^{-e(a,b)},
  \]
  where $f \in \QQ[X^{\pm 1},Y_1^{\pm 1},\dotsc,Y_l^{\pm 1}]$, $e(a,b)\in
  \NN_0$ and $e(a,b) = 0$ for almost all $(a,b)$.
\item\label{d:MMb}
  $W(X,X^{-s_1},\dotsc,X^{-s_l}) \in \QQ(s_1,\dotsc,s_l)\llb X-1\rrb$
  is a \itemph{power series} in $X-1$ (instead of merely a Laurent series).
\end{enumerate}
\end{defn}

\begin{defn}
Let $\red{W}$ denote the image of $W\in \MM[X,Y_1,\dotsc,Y_l]$ under
\begin{align*}
\MM[X,Y_1,\dotsc,Y_l] & \into \QQ(s_1,\dotsc,s_l)\llb X-1\rrb \onto
\QQ(s_1,\dotsc,s_l), \\
f & \mapsto f(X,X^{-s_1},\dotsc,X^{-s_l}) \bmod {(X-1)}.
\end{align*}
\end{defn}

\begin{notation*}
In case $l = 1$, we just write $Y = Y_1$ and $s = s_1$.
\end{notation*}

The following generalisation of results from \cite{DL92}
provides the key to defining topological zeta functions via ``explicit formulae'' 
(in the sense of \cite[\S 3]{Den91a}) such as those in Theorem~\ref{thm:SV_ThmA}.

\begin{thm}
  \label{thm:pre_rigidity}
  Let $V_1,\dotsc,V_r$ be separated $\fo$-schemes of finite type,
  let $W_1,\dotsc,W_r\in \MM[X,Y_1,\dotsc,Y_l]$, and let $S\subset \Places_k$ be
  finite.
  If
  $\sum\limits_{i=1}^r \noof{V_i(\fO_K/\fP_K)} \dtimes W_i(q_K,Y_1,\dotsc,Y_l) = 0$
  for all $v\in \Places_k\setminus S$  and all finite unramified extensions $K/k_v$,
  then $\sum\limits_{i=1}^r \Euler(V_i(\CC)) \dtimes \red{W_i} = 0$.
\end{thm}
\begin{proof}
  This follows from \cite[Thm~5.12]{topzeta}.
\end{proof}

The topological Euler characteristics $\Euler(V_i(\CC))$ are taken with
respect to an arbitrary embedding $k \incl \CC$.
By interpreting these numbers as limits  of $\noof{V_i(\fO_K/\fP_K)}$ 
as ``$q_K\to 1$'' (see~\cite[\S 5.3]{topzeta}), Theorem~\ref{thm:pre_rigidity} fits
the informal description from  the introduction.

\subsection{Defining topological representation zeta functions of unipotent groups}

\begin{lemma}
  \label{lem:good_shape}
  In Theorem~\ref{thm:SV_ThmA}, we may assume that $W_1,\dotsc,W_r \in \MM[X,Y]$.
\end{lemma}

Before giving a proof of Lemma~\ref{lem:good_shape}, let us use it to define
topological representation zeta functions.
Let $\GG$ be a unipotent algebraic group over $k$
and let $\G$ be an arbitrary $\fo$-form of $\GG$.
By Theorem~\ref{thm:SV_ThmA} and Lemma~\ref{lem:good_shape}, there are
$\fo$-schemes $V_1,\dotsc,V_r$ 
(separated, of finite type), rational functions
$W_1,\dotsc,W_r\in \MM[X,Y]$, and a finite $S\subset \Places_k$ such that
\begin{align}
\label{eq:zeta_GOK}
\zeta_{\G(\fO_K)}(s) & = \sum_{i=1}^r \noof{V_i(\fO_K/\fP_K)} \dtimes W_i(q_K^{\phantom {-s}}\!\!, \,q_K^{-s})
\end{align}
is an identity of analytic functions
for all $v\in \Places_k\setminus S$ and all finite extensions $K/k_v$.
For all these $v$ and $K$, the rational function
$\sum_{i=1}^r \noof{V_i(\fO_K/\fP_K)} \dtimes W_i(q_K,Y) \in \QQ(Y)$ then only depends on
$\G(\fO_K)$.
Hence, by Theorem~\ref{thm:pre_rigidity}, the following definition only depends on
$\GG$ and not on the choices of $\G$ or the $V_1,\dotsc,V_r$ and
$W_1,\dotsc,W_r$.

\begin{defn}
  \label{d:topzeta}
  Let $\GG$ be a unipotent algebraic group over $k$.
  Let $\G$, $V_1,\dotsc,V_r$, and $W_1,\dotsc,W_r$ be as above.
  The \emph{topological representation zeta function} of $\GG$ is
  \[
  \zeta_{\GG,\topo}(s) := \sum_{i=1}^r \Euler(V_i(\CC))\dtimes \red{W_i} \in \QQ(s).
  \]
\end{defn}

\begin{ex}
  \label{ex:basic}
  \quad
  \begin{enumerate}
  \item
    (See~\S\ref{s:intro}.)
    A $\ZZ$-form $\H$ of the Heisenberg group  $\HH$ over $\QQ$ is given by
    $\H(R) = \Bigl[\begin{smallmatrix}1 & R & R \\ 0 & 1 & R \\ 0 & 0 & 1\end{smallmatrix}\Bigr] \le
    \GL_3(R)$
    for commutative rings $R$.
    As a special case of \cite[Thm~B]{SV14},
    Voll and Stasinski proved that for all primes $p$ and all finite
    extensions $K/\QQ_p$, we have
    $\zeta_{\H(\fO_K)}(s) = \frac{1-q_K^{-s}}{1-q_K^{1-s}} =
    W(q_K^{\phantom{-s}}\!\!,q_K^{-s})$,
    where $W := \frac{1-Y}{1-XY} \in \MM[X,Y]$.
    We conclude that $\zeta_{\HH, \topo}(s) = \red{W} = s/(s-1)$.
  \item
    du~Sautoy~\cite{dS01} constructed a nilpotent Lie ring $\fg$ of additive
    rank~$9$ and class~$2$ whose ideal ideal zeta function depends on the number 
    $\noof{E(\FF_p)}$, where $E \subset \mathbf P^2_{\ZZ}$ is defined
    by $Y^2 Z = X^3-XZ^2$ (so $E\otimes_{\ZZ} \QQ$ is an elliptic curve).
    Let $\GG$ be the unipotent $\QQ$-group corresponding to $\fg \otimes_\ZZ \QQ$.
    Snocken~\cite[Thm~5.4]{Snocken} computed the local representation zeta functions of
    a $\ZZ$-form, $\G$ say, of $\GG$ (in our terminology) and found
    them to be given by
    $\zeta_{\G(\ZZ_p)}(s)  = 
    W_1(p,p^{-s}) + \noof{E(\FF_p)} \dtimes W_2(p,p^{-s})$,
    where $W_1 := \frac{1-Y^3}{1-X^3Y^3}$
    and $W_2 := \frac{(X-1)(Y-1)Y}{(1-X^2Y^2)(1-X^3Y^3)}$ are elements of
    $\MM[X,Y]$.
    Looking at Snocken's calculation, 
    it is easy to see that at least for odd $p$, we may replace ``$\ZZ_p$'' by
    ``$\fO_K$''   and ``$p$'' by ``$q_K$'' in this identity for any finite extension $K/\QQ_p$.
    Since $\Euler(E(\CC)) = 0$, we conclude that $\zeta_{\GG,\topo}(s) =
    s/(s-1)$, which coincides with $\zeta_{\HH,\topo}(s)$ from the
    preceding example.

    While topological representation zeta functions of unipotent groups
    are coarser invariants than their $p$-adic counterparts, note that in our
    example, $\zeta_{\GG,\topo}(s)$ nonetheless retains essential analytic properties of
    the $p$-adic zeta functions $\zeta_{\G(\fO_K)}(s)$.
    Namely, if $s_0\in \CC$ is a pole ($s_0=1$) or zero ($s_0=0$) of
    $\zeta_{\GG,\topo}(s)$, then $s_0$ is also a pole or zero of
    $\zeta_{\G(\fO_K)}(s)$, respectively.
    We note that for poles of topological zeta functions, a similar
    behaviour is a general phenomenon by \cite[Thm~2.2]{DL92} which 
    provides a key motivation for studying topological zeta functions in the first place.
  \end{enumerate}
\end{ex}

The preceding two examples are misleading in that both the underlying
$p$-adic computations and the final results are quite simple.
In general, due to the reliance (via \cite[Thm~2.1]{Vol10}) of their proofs on
the usually impractical step of constructing  a principalisation of ideals,
Theorem~\ref{thm:SV_ThmA} and Lemma~\ref{lem:good_shape} do not 
provide us with a means of explicitly computing $p$-adic or topological
representation zeta functions.
Despite such theoretical obstacles, in~\S\ref{s:compute}, we develop a
practical method for explicitly computing  $\zeta_{\GG,\topo}(s)$ at least in
favourable situations.
This will allow us to determine a substantial number of interesting examples of these zeta 
functions, see~\S\ref{s:examples}.
 Our computations of topological zeta functions include many
examples with unknown associated $p$-adic zeta functions.

\begin{proof}[Proof of Lemma~\ref{lem:good_shape}]
  First recall the main ingredients featuring in the proof of
  Theorem~\ref{thm:SV_ThmA}. 
  Clearly, we may assume that $\GG$ is non-abelian.
  As we are free to enlarge $S$,
  we may assume that $\G$ is an $\fo$-form 
  constructed from a suitable nilpotent Lie $\fo$-algebra
  in \cite[\S 2.1.2]{SV14}, see Remark~\ref{r:SV_form}. 
  As explained in \cite[\S\S 2.2--2.3]{SV14},
  for almost all $v\in \Places_k$ and all finite extensions $K/k_v$,
  the zeta function $\zeta_{\G(\fO_K)}(s)$ then coincides
  with the Poincar\'e series $\cP_{\mathcal R,\mathcal S,\fO_K}(s)$
  (see~\cite[Prop.~2.9]{SV14})
  attached to certain matrices $\mathcal R$ and $\mathcal S$ of linear forms
  in $n = \dim([\GG,\GG])> 0$ variables over $\fo$.
  The function $\cP_{\mathcal R,\mathcal S,\fO_K}(s)$ is a univariate
  specialisation of $P_{\cR,\cS,K}(\mathbf r, \mathbf
  s)$ from \cite[p.~1199]{Vol10} (cf.~\cite[p.~517]{SV14}).
  Using \cite[p.~1201]{Vol10},
  \[
  \bigl(1-q_K^{-1}\bigr)^2
  \bigl(1-q_K^{-2}\bigr)
  \dotsb
  \bigl(1-q_K^{-(n-1)}\bigr)
  (P_{\mathcal R,\mathcal S,K}(\mathbf r, \mathbf s)-1) =
  Z_{\{1\}}\bigl(-\mathbf r, -\mathbf s,\!\sum\mathbf r +\sum\mathbf s - n -1\bigr),
  \]
  where $Z_{\{1\}}$ is defined in \cite[p.~1194]{Vol10} 
  as a specialisation of \cite[Eqn~(6), p.~1191]{Vol10}.
  Using \cite[Thm~2.2]{Vol10} and Voll's notation, after further excluding
  finitely many places of $k$,
  \begin{equation}
    \label{eq:minus1}
    \zeta_{\G(\fO_K)}(s)-1 =
    F(q_K) \dtimes \sum_{U\subset T} c_U(q_K) \dtimes (q_K-1)^{\noof U +1}\dtimes \Xi_{U,\{1\}}(q_K,\dotsc),
  \end{equation}
  where
  $F(X) =
  X^{-1-n(n-1)/2}\dtimes \prod_{j=2}^{n-1} \frac{1-X^{-1}}{1-X^{-j}} 
  \in \MM[X,Y]$.
  We did not spell out the specific substitutions applied to the 
  functions $\Xi_{U,\{1\}}$ as these are without relevance here;
  they will however become important in the proof of
  Proposition~\ref{prop:infinity} below.
  By construction, the $c_U(q_K)$ are the numbers of
  $(\fO_K/\fP_K)$-rational points of certain quasi-projective $\fo$-schemes.
  Using a suitable triangulation, as indicated in the proof
  of \cite[Prop.~2.1, pp.~1196--1197]{Vol10}, each 
  $\Xi_{U,\{1\}}(q_K,\dotsc)$ can be written as a $\ZZ$-linear combination of rational functions in
  $q_K^{\phantom{-s}}\!\!$ and $q_K^{-s}$, each of which
  is obtained by a monomial substitution from the generating function
  enumerating lattice points in a rational polyhedral cone of dimension at most
  $\noof U + 1$.
  Hence, using~\cite[Lem.~6.11]{topzeta2}
  (with the first column
  of $A$ replaced by the
  transpose of $((\nu_u)_{u\in U},1)\in \NN^U\times\NN$), 
  if $W_U \in \QQ(X,Y)$ with $W_U(q_K^{\phantom
    s},q_K^{-s}) = (q_K-1)^{\noof U+1}\dtimes \Xi_{U,\{1\}}(q_K,\dotsc)$,
  then $W_U\in \MM[X,Y]$ whence the claim follows.
\end{proof}

\subsection{Topological representation zeta functions from perfect Lie algebras}
\label{ss:perfect}

Recall that $r_n(G)$ denotes the number of equivalence
classes (in the usual sense) of continuous irreducible representations
$G\to\GL_n(\CC)$ of a topological group $G$.
It is well-known that if $G$ is compact and $p$-adic analytic, then
$r_n(G)<\infty$ for all $n\ge 1$ if and only if $G$ is FAb, i.e.~every open
subgroup of $G$ has finite abelianisation. 
In that case, let $\zeta_G^{\irr}(s) = \sum_{n=1}^\infty r_n(G)n^{-s}$;
note that in contrast to $\zeta_G(s)$ from above,
we do not take into account twisting.
In~\cite{AKOV13}, uniformity results for families of 
representation zeta functions $\zeta_G^{\irr}(s)$ derived from a fixed global Lie
lattice were obtained.
We now briefly describe how associated topological zeta functions can be
obtained in the same context.

Let $\bm\fg$ be a finite-dimensional perfect Lie $k$-algebra and choose an
arbitrary $\fo$-form $\fg$ of $\bm\fg$.
By~\cite[\S 2.1]{AKOV13}, there exists a finite set $S\subset \Places_k$ such that
for all $v\in \Places_k\setminus S$ and all \itemph{unramified} finite extensions $K/k_v$,
we may naturally endow $\fP_K (\fg\otimes_{\fo}\fO_K)$ with the structure of a
FAb $K$-analytic pro-$p_v$ group, denoted $\G^1(\fO_K)$.
The analysis of the zeta functions $\zeta_{\G^1(\fO_K)}^{\irr}(s)$
in \cite[\S\S 3--4]{AKOV13} and the results from \cite{SV14} used above
rely on the same core ingredients.
In particular, the $p$-adic integrals in \cite[Cor.~3.7]{AKOV13} and
\cite[Cor.~2.11]{SV14} are of the same shape and the key arguments from the proof of
Lemma~\ref{lem:good_shape} apply to both cases.
It thus follows from \cite[\S 4]{AKOV13} that after replacing $S$ by a finite superset,
there are separated $\fo$-schemes $V_1,\dotsc,V_r$ of finite type and
$W_1,\dotsc,W_r\in \MM[X,Y]$ such that
$\zeta_{\G^1(\fO_K)}(s) = \sum_{i=1}^r \noof{V_i(\fO_K/\fP_K)} \dtimes W_i(q_K^{\phantom{-s}}\!\!,q_K^{-s})$
for all $v\in \Places_k\setminus S$ and all unramified finite extensions $K/k_v$.
By Theorem~\ref{thm:pre_rigidity}, we may then unambiguously define the
\emph{topological representation zeta function} of $\bm\fg$ to be
$\zeta_{\bm\fg,\topo}^{\irr}(s) = \sum_{i=1}^r \Euler(V_i(\CC))\dtimes
\red{W_i}$.
For example, the explicit $p$-adic formulae in \cite[Thm~1.2]{AKOV12} and
\cite[Thm~1.4]{AKOV12} show that
$\zeta_{\mathfrak{sl}_2(\QQ),\topo}^{\irr} = \frac{s+2}{s-1}$ and
$\zeta_{\mathfrak{sl}_3(\QQ),\topo}^{\irr} = \frac{6(s+1)(s+2)}{(3s-2)(2s-1)}$.\\

While the focus of the present article is on the unipotent case,
we note that statements analogous to the properties of topological
representation zeta functions of unipotent groups derived in
\S\ref{s:properties} below also hold for the functions
$\zeta_{\bm\fg,\topo}^{\irr}(s)$ (see~Remark~\ref{rem:perfect_properties}).

\section{Fundamental properties}
\label{s:properties}

We derive some basic facts about topological representation zeta functions.

\subsection*{Direct products}

As we will now see, it suffices to consider topological representation zeta
functions of unipotent groups associated with $\oplus$-indecomposable nilpotent Lie algebras. 

\begin{lemma}[{Cf.~\cite[\S 4.1.1]{Snocken}}]
  $\zeta_{G_1\times G_2}(s) = \zeta_{G_1}(s)\dtimes \zeta_{G_2}(s)$
  for all finitely generated nilpotent pro-$p$ groups $G_1$ and $G_2$.
\end{lemma}
\begin{proof}
  For a topological group $H$,
  let $\Irr(H)$ denote the set of characters of its continuous, irreducible,
  finite-dimensional, complex representations.
  By reducing to the well-known case of finite groups (see e.g.~\cite[8.4.2]{Rob96}),
  we obtain a bijection
  \[
  \Irr(G_1)\times \Irr(G_2) \to \Irr(G_1\times G_2), \quad
  (\chi_1,\chi_2) \mapsto \left(\chi_1\#\chi_2\colon (g_1,g_2) \mapsto \chi_1(g_1) \chi_2(g_2)\right).
  \]
  As $(\chi_1\# \chi_2) \dtimes (\psi_1\#\psi_2) = (\chi_1\psi_1) \#
  (\chi_2\psi_2)$ and $(\chi_1\#\chi_2)(1) = \chi_1(1)\chi_2(1)$
  for $\chi_i,\psi_i \in \Irr(G_i)$, we have
  $\tilde r_n(G_1\times G_2) = \sum\limits_{\divides d n} \tilde r_d(G_1) \tilde
  r_{n/d}(G_2)$
  whence the claim follows.
\end{proof}

\begin{cor}
  \label{cor:product}
  $\zeta_{\GG\times_k \HH,\topo}(s) =
  \zeta_{\GG,\topo}(s) \dtimes \zeta_{\HH,\topo}(s)$ for unipotent
  $k$-groups $\GG$ and $\HH$.
  \qed
\end{cor}

\subsection*{Base extension}

Passing from $p$-adic to topological representation zeta functions removes the
possibly very subtle arithmetic dependence on the defining number field $k$
indicated by Theorem~\ref{thm:SV_ThmA}.

\begin{prop}
  \label{prop:base}
  Let $\GG$ be a unipotent $k$-group, let $\tilde k/k$ be a
  finite extension, and let $\tilde\GG$ be the $\tilde k$-group obtained from
  $\GG$  via base extension.
  Then $\zeta_{\GG{\phantom{\tilde\GG}}\!\!\!\!\!,\topo}(s) = \zeta_{\tilde\GG,\topo}(s)$.
\end{prop}
\begin{proof}
  Immediate from the stability under base extension in Theorem~\ref{thm:SV_ThmA}.
\end{proof}

We may thus unambiguously define the topological representation zeta function of
a unipotent algebraic group over an algebraic closure of $\QQ$ to be that of an
arbitrary $k$-form for a suitable number field $k$.

\subsection*{Restriction of scalars}

Let $\mathrm{res}_{\tilde k/k}(\blank)$ denote Weil restriction from
affine $\tilde k$-groups to affine $k$-groups.
\begin{prop}
  \label{prop:res}
  Let $\tilde k/k$ be an extension of number fields and let $\tilde\GG$ be a unipotent
  $\tilde k$-group.
 Then $\zeta_{\mathrm{res}_{\tilde k/k}(\tilde\GG),\topo}(s) =
  \zeta_{\tilde\GG, \phantom{\mathrm{res}_{\tilde k/k}(\tilde\GG)}\!\!\!\!\!\!\!\!\!\!\!\!\!\!\!\!\!\!\!\!\!\!\!\!\!\!\!\topo}(s)^{\idx{\tilde k:k}}$.
\end{prop}
\begin{proof}
  Let $l$ be a normal closure of $\tilde k$ over $k$ and
  let $\Sigma$ be the set of $k$-embeddings $\tilde k \to l$;
  note that $\noof \Sigma = \idx{\tilde k:k}$.
  For $\sigma\in\Sigma$, let
  $\tilde\GG^{\sigma}$ be the $l$-group obtained from $\tilde\GG$ by base change along
  $\sigma$. 
  Then 
  $\mathrm{res}_{\tilde k/k}(\tilde\GG) \otimes_k l \approx_l
  \prod_{\sigma\in\Sigma} \tilde\GG^\sigma$, see \cite[Ch.~V, \S 5.7]{milneAGS} and cf.~\cite[12.4.5]{Spr98}.
  The claim now follows from Corollary~\ref{cor:product} and Proposition~\ref{prop:base}.
\end{proof}

\subsection*{Behaviour at infinity}

If $G$ is a finitely generated nilpotent pro-$p$ group, then 
$\lim\limits_{s\to+\infty}\zeta_G(s) = \tilde r_1(G) = 1$.
This entirely trivial fact on the $p$-adic side survives passing to topological zeta functions:

\begin{prop}
  \label{prop:infinity}
  $\lim\limits_{s\to \infty}\zeta_{\GG,\topo}(s)=1$ for
  every unipotent $k$-group $\GG$.
\end{prop}
\begin{rem}
  Unlike, say, Corollary~\ref{cor:product},
  Proposition~\ref{prop:infinity} is not merely a formal consequence 
  of the analogous $p$-adic statement.
  As an illustration,
  consider
  \[
  W(X,Y) =\frac{(X^2Y^6 + X^2Y^3 - 4XY^3 + Y^3 + 1)(1-Y)^2}{(1 - X^4Y^4)(1-X^2Y^2)(1-XY)^2}
  \in \MM[X,Y].
  \]
  For $q>1$, we have
  $W(q,q^{-s}) \to 1$ as $s\to +\infty$ 
  and yet, the associated ``topological zeta function''
  $\red{W} = \frac{(9s^2 - 6s + 2)s^2}{8(s - 1)^4}$
  takes the value $9/8$ at infinity.
  Note that $W(X,Y)$ satisfies the functional equation 
  $W(X^{-1},Y^{-1}) = X^6 \dtimes W(X,Y)$
  associated with representation zeta functions of unipotent groups with
  $6$-dimensional derived groups, see \cite[Thm~A]{SV14}.
\end{rem}
\begin{proof}[Proof of Proposition~\ref{prop:infinity}]
  Let $\GG$ be non-abelian.
  We continue to use the setting and notation from the proof of
  Lemma~\ref{lem:good_shape}.
  We often suppress references to the set $I$ from \cite[\S 2.1]{Vol10} and its elements as $I= \{1\}$ is
  fixed. 
  Using the triangulation argument from above and noting that $\red F = \frac
  1{(n-1)!}$, we may write
  \begin{equation}
    \label{eq:minus1_v2}
    \zeta_{\GG,\topo}(s)-1 = \frac{1}{(n-1)!}
    \sum_{U\subset T}
    \Euler(V_U(\CC)) \dtimes
    \sum_{a\in A_U}
    \gamma^a_U \red{W_U^a},
  \end{equation}
  where $V_U := E_U\setminus \cup_{U' \supsetneq U} E_{U'} $ gives rise
  to the $c_U(q_K)$ in \eqref{eq:minus1} as
  in~\cite[Thm~2.2]{Vol10},
  each $A_U$ is finite, and for $a\in A_U$, each
  $\cC_U^a\subset \Orth^U \times \Orth^{\phantom 1}$ is a rational simplicial cone,
  $\gamma^a_U\in \ZZ\setminus\{0\}$, and $W_U^a\in \MM[X,Y]$ such that
  $(q_K-1)^{-\noof U -1}W_U^a(q_K^{\phantom{-s}}\!\!,q_K^{-s})$ is obtained from
  the generating function enumerating lattice points within $\cC_U^a$ 
  by applying a certain specialisation derived from the one mentioned
  (but not yet spelled out) above.
  As in the proof of \cite[Lem.~6.11]{topzeta2}, 
  it suffices to sum over those $a\in A_U$ 
  with $\dim(\cC_U^a) = \noof U + 1$ in \eqref{eq:minus1_v2}.
  Moreover, it is explained in the proof of \cite[Lem.~6.11]{topzeta2} how to
  read off the denominator of each remaining $\red{W^a_U}$ from the extreme rays
  of $\cC_U^a$ and the given univariate specialisation.
  The remainder of the proof is devoted
  to showing that $\deg_s(\red{W_U^a})<0$ for all $U$ and $a$ whence the claim
  follows from \eqref{eq:minus1_v2}.
  As each full-dimensional cone $\cC_U^a$
  contains a vector whose last component is positive, it suffices to prove the
  following:

  \begin{itemize}
  \item[(\textlabel{$\star$}{one_ray_case})]
    Given a $1$-dimensional rational cone $\cC\!\subset\Orth^U\times\Orth^{\phantom 1}$ with
    $\cC\cap(\Orth^U\times\StrictOrth^{\phantom 1})\not= \emptyset$,
    if $W\!\in\! \QQ(X,Y)$ is obtained as above by applying the aforementioned 
    specialisation to the generating function of
    $\cC \cap(\ZZ^U\times\ZZ)$, then $\deg_s(\red{(X-1)W})=-1$. 
  \end{itemize}

  In order to establish (\ref{one_ray_case}), we need
  to consider the relevant specialisations in detail.
  We first recall a description of the $p$-adic integrals giving rise to the
  $\Xi_U$ in \eqref{eq:minus1} via \cite[Thm~2.2]{Vol10}.
  Write $n = \dim([\GG,\GG])$.
  As in \cite[\S 2.2.2]{SV14},
  define matrices $\cR(\YY)$ and $\cS(\YY)$ of linear forms in $\YY =
  (Y_1,\dotsc,Y_n)$ over $\fo$, and
  let $2u$ and $v$ denote the ranks of $\cR(\YY)$ and $\cS(\YY)$ over $k(\YY)$,
  respectively.
  Since $\cR(\YY)$ is anti-symmetric,
  its principal minors are squares in $\fo[\YY]$, see
  \cite[Rem.~3.6]{AKOV13}.
  For $0 \le i \le u$,
  let $F_i(\YY)\subset \fo[\YY]$ consist of square roots of the
  non-zero $2i\times 2i$ principal minors of $\cR(\YY)$.
  In particular, up to possible sign changes, $F_0(\YY) = \{ 1\}$ and $F_1(\YY)$
  consists of the non-zero entries of $\cR(\YY)$.
  For $0\le j \le v$, let $G_j(\YY)$ be the set of non-zero $j\times j$
  minors of $\cS(\YY)$.
  As we mentioned before, $\cP_{\mathcal R,\mathcal S,\fO_K}(s)$ in
  \cite[Prop.~2.9]{SV14} is a univariate specialisation of
  $P_{\cR,\cS,K}(\mathbf r, \mathbf s)$ from \cite[p.~1199]{Vol10}.
  Specifically, $\cP_{\mathcal R,\mathcal S,\fO_K}(s) =
  P_{\cR,\cS,K}(\frac s 2,\dotsc,\frac s 2;1,\dotsc,1)$
  (cf.~\cite[Prop.~3.1, p.~1214]{Vol10}).
  Using the formula for $P_{\cR,\cS,K}(\mathbf r, \mathbf s)-1$ in
  \cite[Eqn~(22), p.~1201]{Vol10},
  write $Z'(s)$ for the function obtained by applying the resulting
  specialisation to the integral $Z(\bm r, \bm s,t)$;
  in other words, $Z'(s) = Z(-\frac s 2,\dotsc,-\frac s
  2;-1,\dotsc,-1;us+v-n-1)$
  (cf.~\cite[Cor.~2.11]{SV14}).
  By exploiting the anti-symmetry of $\cR(\YY)$ as in \cite[\S 3.2]{AKOV13}
  and \cite[\S 2.2.3]{SV14}, and after some rearranging, we may write
  \begin{equation}
    \label{eq:Zsub}
    Z'(s) = 
    \!\!\!\!
    \!\!\!\!
    \!\!\!\!
    \int_{\cP_K\times \GL_n(\fO_K)}
    \!\!\!\! \!\!\!\!
    \abs{x}^{s-n-1}_K
    \prod_{i=2}^u 
    \frac{\abs{x}_K^s \norm{F_{i-1}(\yy^1)}_K^s}{\norm{F_i(\yy^1) \cup xF_{i-1}(\yy^1)}_K^s}
    \prod_{j=1}^v
    \frac{\abs{x}_K\norm{G_{j-1}(\yy^1)}_K}{\norm{G_j(\yy^1) \cup xG_{j-1}(\yy^1)}_K}
    \dd\mu_K(x,\yy),
  \end{equation}
  where $\yy^1$ denotes the first column of $\yy\in \GL_n(\fO_K)$
  and $\mu_K$ is the normalised Haar measure on $\fO_K\times \Mat_n(\fO_K) \approx
  \fO_K^{1+n^2}$.
  In contrast to \cites{Vol10,SV14}, we omitted the factor
  $\frac{\norm{F_{0}(\yy^1)}_K^s}{\norm{F_1(\yy^1) \cup xF_{0}(\yy^1)}_K^{s}}$
  from \eqref{eq:Zsub}.
  Indeed,
  $\norm{F_0(\yy^1)}_K = 1$ trivially and,
  moreover, after excluding finitely many places of $k$ as in~\cite[p.~1215]{Vol10},
  since $\yy^1\not\equiv 0 \bmod{\fP_K}$, the set $F_1(\yy^1)$ always
  contains a $\fP_K$-adic unit whence also $\norm{F_1(\yy^1)\cup xF_0(\yy^1)}_K =
  1$.

  We next observe that for $\yy$ outside of the zero locus (of measure zero) of
  $\prod\bigcup_{i=1}^u F_i$ within $\GL_n(\fO_K)$, the trivial estimate
  $\abs{x}_K \norm{F_{i-1}(\yy^1)}_K \le {\norm{F_i(\yy^1) \cup xF_{i-1}(\yy^1)}_K}$
  shows that $Z'(s)$ not only takes finite values for
  $\Real(s) > n$ 
  but, in addition, $Z'(s) \to 0$ for $\Real(s) \to +\infty$ as $\abs{x}_K<1$.
  By interpreting \eqref{eq:Zsub} as a specialisation
  of $Z_{\{1\}}$ in \cite[p.~1194]{Vol10},
  the $c_U(q_K)$ and $\Xi_U = \Xi_{U,\{1\}}$ in \eqref{eq:minus1}
  are obtained by invoking \cite[Thm~2.2]{Vol10}
  and, finally, the univariate
  specialisation applied to the $s_\kappa$ from \cite{Vol10}
  that we have been referring to throughout this proof is defined.

  To finish the proof, we consider the effect of our univariate
  specialisation on $\Xi_U$.
  Let $(m_u)_{u\in U}$ and $(n_i)_{i\in I} = (n_1)$ as in \cite[Eqn~(9), p.~1193]{Vol10}.
  The denominator of the multivariate generating function of a relatively open
  rational cone coincides with that of its closure, see \cite[Thm~4.5.11]{Sta12}.
  We thus allow $m_u\in \NN_0$ in the following but we still insist that $n_1$
  is positive---indeed, as explained above, we only need to consider full-dimensional cones within
  $\Orth^U\times \Orth^{\phantom 1}$ and these always 
  admit an extreme ray intersecting $\Orth^U\times\StrictOrth^{\phantom U}$.
  Having fixed the $m_u$ and $n_1$, by applying our univariate
  specialisation to the exponent on the right-hand side of \cite[Eqn~(9)]{Vol10},
  we obtain a polynomial, $e \in \ZZ[s]$ say, of degree at most $1$.
  By definition of the numbers $N_{u\kappa\iota}$ in \cite[Eqn~(8), p.~1192]{Vol10},
  the same estimate that we used to establish that $Z'(s) \to 0$ for
  $\Real(s)\to \infty$ shows that $e$ always has degree precisely $1$.
  This implies that (\ref{one_ray_case}) is satisfied and completes the proof.
\end{proof}

\begin{cor}
  \label{cor:degree}
  $\deg_s(\zeta_{\GG,\topo}(s))=0$. \qed
\end{cor}

\subsection*{Location of poles}

The following consequence of the proof of Proposition~\ref{prop:infinity} is
again analogous to the $p$-adic case which is proved in \cite[Thm~4.24]{Snocken}
(and also stated as \cite[Thm~5.4.1]{Ezzat})
for groups of nilpotency class~$2$ and $k = \QQ$;
our proof implicitly also covers the $p$-adic case without any of these
two restrictions.
\begin{prop}
  \label{prop:poles}
  Let $\GG$ be a unipotent $k$-group.
  If $s_0\in \CC$ is a pole of $\zeta_{\GG,\topo}(s)$, then $s_0\in\QQ$ and
  $s_0 \le \dim([\GG,\GG])$.
\end{prop}
\begin{proof}
  The poles of $\red{W}$ are rational
  for each $W\in \MM[X,Y]$
  by \cite[Lem.~5.6(ii)]{topzeta}
  whence $s_0\in \QQ$.
  We continue to use the notation from the proof of Proposition~\ref{prop:infinity}.
  Let $s_1\in \QQ$ with $s_1 > n$.
  Then $Z'(s_1)$ (see \eqref{eq:Zsub}) is finite.
  Noting that $-\sum_u \nu_u m_u \le 0$ in \cite[Eqn~(9), p.~1193]{Vol10} (with
  $I =\{1\}$),
  it follows that if $e\in \ZZ[s]$ is one of the polynomials in the final
  paragraph of the proof of Proposition~\ref{prop:infinity} 
  (where we now also allow $n_1=0$ as long as $((m_u),n_1)\not= (0,0)$), then $e(s_1) < 0$.
  We conclude that each rational function $\red{W_U^a}$ in
  \eqref{eq:minus1_v2} is regular at $s_1$ 
  (as is each $W_U^a(q_K^{\phantom{-s}}\!\!,q_K^{-s})$)
  whence the same is true of
  $\zeta_{\GG,\topo}(s)$.
\end{proof}

\begin{rem}
  \label{rem:perfect_properties}
  Regarding the topological representation zeta functions
  $\zeta_{\bm\fg,\topo}^{\irr}(s)$ from~\S\ref{ss:perfect},
  the statements and proofs of
  Corollaries~\ref{cor:product} and \ref{cor:degree},
  and Propositions~\ref{prop:base}--\ref{prop:infinity} 
  carry over essentially verbatim.
  As for Proposition~\ref{prop:poles},
  if $s_0\in \CC$ is a pole of $\zeta_{\bm\fg,\topo}^{\irr}(s)$,
  then $s_0\in \QQ$ and $s_0\le \dim(\bm\fg)-2$;
  for more refined $p$-adic bounds, see \cite[Thm~1.1]{AKOV12}.
\end{rem}

\section{Computing topological representation zeta functions}
\label{s:compute}

At this point, we have encountered but one non-trivial example of a
topological representation zeta function of a unipotent group, namely $\frac
s{s-1}$ which arises from both $\QQ$-groups in Example~\ref{ex:basic}.
In both cases, we used $p$-adic computations to derive topological
formulae.
We now devise a method for computing topological
representation zeta functions of unipotent groups directly, i.e.~without resorting to a
full $p$-adic computation.
As in the $p$-adic realm
and the cases of previously studied types of topological zeta functions,
a \itemph{general} method for computing these functions seems
to invariably rely on some variation of resolution of singularities.
In order to nonetheless be able to carry out explicit computations, we 
primarily rely on algebro-geometric genericity conditions which may or may not
be satisfied in specific situations.
Despite these theoretical limitations of our method, numerous previously unknown
examples of topological zeta functions (see~\S\ref{s:examples}) demonstrate its
value.
Our method adapts and extends the core ingredients from \cites{topzeta,topzeta2}.

\subsection{Representation data and zeta functions of unipotent groups}

In this subsection, we introduce representation data.
These objects provide a convenient way of encoding certain types of
$p$-adic integrals, in particular those describing local representation zeta
functions attached to unipotent groups as in \cite{Vol10,SV14}.

\paragraph{Representation data.}
We first recall some notions from convex geometry.
By a \emph{cone} in $\RR^n$, we always mean a polyhedral one.
As in \cite[\S 3.1]{topzeta}, a \emph{half-open cone} in $\RR^n$ is
a set $\cC_0 = \cC \setminus(\cC_1\cup\dotsb\cup\cC_r)$ for a cone
$\cC\subset\RR^n$ and faces $\cC_1,\dotsc,\cC_r\subset \cC$.
We say that $\cC_0$ is \emph{rational} if $\cC$ is
rational in the usual sense or if $\cC_0 =\emptyset$.
Define the \emph{dual} of a set $A\subset \RR^n$ by $A^* := \{ \omega\in \RR^n :
\bil\alpha\omega \ge 0 \text{ for all } \alpha\in A\}$, where
$\bil{\dtimes}{\dtimes}$ denotes the standard inner product.
The dual of a finite union of rational half-open cones is a rational cone.
For a cone $\cC\subset \RR^n$, let $k[\cC\cap\ZZ^n]$ denote
the $k$-subalgebra of $k[\XX^{\pm 1}] := k[X_1^{\pm 1},\dotsc,X_n^{\pm 1}]$
spanned by all $\XX^\omega := X_1^{\omega_1}\dotsb X_n^{\omega_n}$ for $\omega\in \cC\cap\ZZ^n$. 

\begin{notation}
  For $G\subset k[\XX^{\pm 1}]$ and $\gamma\colon G \to \ZZ^n$,
  write
  $\XX^\gamma G := \bigl\{ \XX^{\gamma(g)} \dtimes g : g\in G\bigr\}$.
\end{notation}

\begin{defn}
  \label{d:datum}
  A \emph{representation datum} over $k$ (in $n$ variables) is a quadruple
  $$\cR = (\cD, F, \alpha, \Lambda)$$ consisting of
  \begin{enumerate}
  \item
    $\cD =\bigcup_{\delta\in\Delta}\cC_0^\delta$ (finite disjoint union) for
    rational half-open cones $\cC_0^\delta\subset\Orth^n$,
  \item
    $F = F_1\cup\dotsb\cup F_m$ for not necessarily disjoint
    non-empty finite sets $F_i \subset k[\XX^{\pm 1}]$,
  \item
    $\alpha = (\alpha_1,\dotsc,\alpha_m)$ for $\alpha_i\colon F_i \to \ZZ^n$
    with $\XX^{\alpha_i} F_i\subset k[\cD^*\cap\ZZ^n]$ ($i=1,\dotsc,m$), and
  \item
    $\Lambda \subset \{1,\dotsc,m,\infty\}^2$.
  \end{enumerate}
\end{defn}

By minor abuse of notation,
we consider the specifications of the distinguished sets $\cC_0^\delta$ and
$F_i$ to be part of the definition of a representation datum.
In the following, we always assume that $0\not\in F$
and $(\infty,\infty)\not\in \Lambda$.
We write $F_\infty = \{0\}$ and let $\alpha_\infty\colon F_\infty\to \ZZ^n$ with $\alpha_\infty(0)=0$.

\paragraph{Associated integrals.}
For a non-Archimedean local field $K$, let $\nu_K$ denote the valuation
on $K$ with $\nu_K(K^\times) = \ZZ$.
We let $\abs{\dtimes}_K$ denote the absolute value on $K$ defined by $\abs{x}_K
=q_K^{-\nu_K(x)}$ for $x\in K$.
For $\xx = (x_1,\dotsc,x_n)\in K^n$, write $\nu_K(\xx) =
(\nu_K(x_1),\dotsc,\nu_K(x_n))$.
Let $\Torus^n = \Spec(\ZZ[X_1^{\pm 1},\dotsc,X_n^{\pm 1}])$.
If $R$ is a commutative ring,
identify $\Torus^n(R) = (R^\times)^n$.
If $A\subset K$ is finite and non-empty, we again write $\norm{A}_K = \max(\abs{a}_K : a\in A)$.
\begin{notation}
  \label{not:AK}
  For $A\subset\RR^n$, write $A(K) = \{ \xx\in K^n : \nu_K(\xx)\in A\}$.
\end{notation}

\begin{defn}
  \label{d:integral}
  Let $\cR = (\cD,F,\alpha,\Lambda)$ be a representation datum as in
  Definition~\ref{d:datum}.
  Recall that $F_\infty=\{0\}$ and $\alpha_\infty =0$.
  For a $p$-adic field $K\supset k$, define the ``zeta function''
  \begin{equation}
  \label{eq:integral}
  \Zeta^\cR_K(\bm s) =
    \int_{\cD(K)\times\fP_K}
    \prod_{\lambda = (i,j)\in \Lambda} \norm{\xx^{\alpha_i}F_i(\xx) \cup \xx^{\alpha_j}yF_j(\xx)}_K^{s_\lambda} \dd\mu_K(\xx,y),
  \end{equation}
  where $\bm s = (s_\lambda)_{\lambda\in\Lambda}$ consists
  of complex variables and $\mu_K$ is the normalised Haar measure.
\end{defn}

\begin{rem}
  \label{r:integral}
   \quad
   \begin{enumerate}
   \item
     \label{r:integral1}
     By Proposition~\ref{prop:absint} below,
     if $f\in k[\cD^*\cap\ZZ^n]$ and $\xx \in \cD(K)$, then, unless $\fo\cap\fP_K$ belongs to a
     finite exceptional set, we have $\abs{f(\xx)}_K\le 1$.
     Hence, the assumption $\XX^{\alpha_i}F_i\subset k[\cD^*\cap\ZZ^n]$ in
     Definition~\ref{d:datum} guarantees the finiteness of \eqref{eq:integral}
     at the very least when $\Real(s_\lambda) \ge 0$ for all $\lambda\in
     \Lambda$.
   \item
     Note that $\Zeta^\cR_K(\bm s)$ only depends
     on the sets $\hat F_i := \XX^{\alpha_i} F_i$ and not on the $F_i$ themselves.
     Conversely, given $\hat F_1,\dotsc,\hat F_m\subset k[\cD^*\cap\ZZ^n]$, there are
     various ways of writing $\hat F_i = \XX^{\alpha_i} F_i$ for $F_i$ and $\alpha_i$
     as above.
     In~Theorem~\ref{thm:topregular}, we will derive an explicit
     formula for a topological zeta function associated with
     $\Zeta^{\cR}_K(\bm s)$ under non-degeneracy assumptions on
     $F = F_1\cup\dotsb\cup F_m$.
     The advantage of representing $\hat F_i$ by a pair $(F_i,\alpha_i)$ 
     is that by carefully choosing the latter, we may be able to produce
     overlaps between the $F_i$ and resolve a seemingly
     degenerate situation.
 \end{enumerate}
\end{rem}

\paragraph{Connection with representation zeta functions.}
Our interest in representation data is due to the following result
which also explains our choice of terminology.

\begin{thm}[{Cf.~\cite[Cor.~2.11]{SV14} and \cite[\S 3.4]{Vol10}}]
  \label{thm:repint}
  Let $\GG$ be a unipotent algebraic group over $k$ and let $n =
  \dim([\GG,\GG])$.
  Then there exist an explicit representation datum $\cR= (\cD,F,\alpha,\Lambda)$
  in $n$ variables as in Definition~\ref{d:datum} and explicit numbers
  $a_\lambda,b_\lambda\in\ZZ$ with the following property:
  if $\G$ is any  $\fo$-form of $\GG$, 
  there exists a finite set  $S\subset \Places_k$ such that
  \[
  \zeta_{\G(\fO_K)}(s) = 1 + (1-q_K^{-1})^{-1} \Zeta^\cR_K(a_\lambda s
  - b_\lambda)_{\lambda\in\Lambda}
  \]
  for all $v\in \Places_k\setminus S$ and all finite extensions $K/k_v$.
  Moreover, we may assume that $\cD = \partial \Orth^n$
  (so that $k[\cD^*\cap\ZZ^n]=k[\XX]$)
  and that $F$ consists of homogeneous polynomials.
\end{thm}

We already made implicit use of the main ingredients of Theorem~\ref{thm:repint}
in the proof of Proposition~\ref{prop:infinity} even though the integrals there
were of a slightly different shape, see~\cite[\S 4.1.3]{AKOV13}.
Of course, our formalism of representation data did not feature in
\cites{Vol10,SV14}, but
it is easy to interpret their integrals as
specialisations of \eqref{eq:integral}.
Using \cite[\S\S 2.2.2--2.2.3]{SV14}, the representation datum
$\cR$ and family $(a_\lambda,b_\lambda)_{\lambda\in\Lambda}$
in Theorem~\ref{thm:repint} can be constructed from the
structure constants of the Lie algebra of $\GG$ with respect to a suitable basis.

\paragraph{Associated topological zeta functions.}
Let us say that a representation datum $\cR$ is \emph{$(-1)$-expandable}
if there are separated $\fo$-schemes of finite type $V_1,\dotsc,V_r$ and
rational functions $W_1,\dotsc,W_r\in \MM\left[X,(Y_\lambda)_{\lambda\in\Lambda}\right]$
(see~Definition~\ref{d:MM}) such that for almost all $v\in \Places_k$ and all
finite extensions $K/k_v$,
\begin{equation}
  \label{eq:minus1exp}
  (1-q_K^{-1})^{-1} \Zeta^\cR_K(\bm s) = \sum_{i=1}^r
  \noof{V_i(\fO_K/\fP_K)}\dtimes W_i\Bigl(q_K^{\phantom{s_\lambda}}\!,
   \bigl(q_K^{-s_\lambda}\bigr)_{\lambda\in\Lambda}\Bigr),
\end{equation}
where $\bm s = (s_\lambda)_{\lambda\in\Lambda}$;
the ``$(-1)$' in $(-1)$-expandability reflects 
the exponent on the left-hand side of \eqref{eq:minus1exp}.
If $\cR$ is $(-1)$-expandable, then
Theorem~\ref{thm:pre_rigidity} allows us to unambiguously
define the \emph{topological zeta function} associated
with $\cR$ to be
\[
\Zeta^\cR_{\topo}(\ess)
:=
\sum_{i=1}^r\Euler(V_i(\CC)) \dtimes \red{W_i} \in \QQ(\ess).
\]

\begin{ex}
  \label{ex:reps_via_repdatum}
  If $\cR$ is obtained using Theorem~\ref{thm:repint},
  then \cite[Thm~2.2]{Vol10}, \cite[Eqn~(4.8)]{AKOV13}, and the proof of
  Lemma~\ref{lem:good_shape} show that $\cR$ is $(-1)$-expandable. 
  Moreover, passing to topological zeta functions commutes with our univariate
  specialisations: by~\cite[Rem.~5.15]{topzeta},
  if the $(a_\lambda,b_\lambda)_{\lambda\in\Lambda}$ in Theorem~\ref{thm:repint}
  are chosen corresponding to \cite[Cor.~2.11]{SV14}, then
  \[
  \zeta_{\GG,\topo}(s) -1 =
  \Zeta^\cR_{\topo}(a_\lambda s - b_\lambda)_{\lambda\in\Lambda}.
  \]
\end{ex}

\subsection{Integer-valued Laurent polynomials}

The techniques for computing with $p$-adic integrals developed in
\cites{topzeta,topzeta2} and extended here have their origin in toric
geometry.
The appearance of the algebra $k[\cD^*\cap\ZZ^n]$ in Definition~\ref{d:datum} is
therefore perhaps not surprising since
$k[\cC^*\cap\ZZ^n]$ is the coordinate ring of the affine toric $k$-variety attached to
a rational cone $\cC\subset\RR^n$ (see~\cite[Thm~1.2.18]{CLS11}).
We now derive a simple arithmetic characterisation of $k[\cC^*\cap\ZZ^n]$
which, in particular, implies the finiteness of the integrals in
Definition~\ref{d:integral} at least for $\Real(s_\lambda)\ge 0$.

Let $f\in k[\XX^{\pm 1}]$ and write $f =
\sum_{\alpha\in\ZZ^n}c_\alpha \XX^\alpha$, where $c_\alpha\in k$, almost all of
which are zero.
Recall that the \emph{support} of $f$ is $\supp(f) := \{ \alpha\in \ZZ^n:
c_\alpha\not= 0\}$.
The convex hull of $\supp(f) \subset \RR^n$ is
the \emph{Newton polytope} $\Newton(f)$  of $f$.
For a $p$-adic field $K\supset k$ and $M \subset \Torus^n(K)$, 
we say that $f$ is \emph{integer-valued} on $M$
if $\nu_K(f(\xx))\ge 0$ for all $\xx\in M$. 
\begin{prop}
  \label{prop:absint}
  Let $k$ be a number field, $f\in k[X_1^{\pm 1},\dotsc,X_n^{\pm 1}]$,
  and let $\cC\subset\RR^n$ be a rational cone.
  Then the following are equivalent:
  \begin{enumerate}
  \item
  \label{p:absint1}
    $\Newton(f) \subset \cC^*$.
  \item
  \label{p:absint2}
    There exists a finite set $S\subset \Places_k$ such that
    $f$ is integer-valued on $\cC(K)$ (see Notation~\ref{not:AK}) for all
    $v\in \Places_k\setminus S$ and  all finite extensions $K/k_v$.
  \end{enumerate}
\end{prop}

\begin{proof}
  Let $f\not= 0$. 
  For $c\in K$, $\alpha\in\ZZ^n$, and $\xx\in \Torus^n(K)$ with
  $\nu_K(\xx)=\omega$, we have
  $\nu_K(c\xx^\alpha) = \nu_K(c) + \bil\alpha\omega$.
  Hence, if~(\ref{p:absint1}) holds
  and the coefficients of $f$ are $v$-adic integers for $v\in \Places_k$,
  then $f$ is integer-valued on $\cC(K)$ for all finite extensions $K/k_v$.
  Suppose that~(\ref{p:absint2}) holds.
  Choose $v\in \Places_k\setminus S$ such that all non-zero coefficients of
  $f$ are $v$-adic units and let $\omega\in\cC\cap\ZZ^n$.
  Then the initial form $\init_\omega(f)$ (see \S\ref{sss:balreg}) has non-zero image in
  $\RF_v[\XX^{\pm 1}]$.
  Hence, there exist a finite extension $K/k_v$ and $\uu
  \in \Torus^n(\fO_K)$
  with $\init_\omega(f)(\uu)\not\equiv 0\bmod{\fP_K}$. 
  Choose $\pi \in \fP_K^{\phantom 1}\setminus\fP_K^2$, let $\xx :=
  (\pi^{\omega_1}u_1,\dotsc,\pi^{\omega_n}u_n)$ 
  and $\alpha\in\supp(\init_\omega(f))$.
  Then $f(\xx) = \pi^{\bil\alpha\omega} (\init_\omega(f)(\uu) + \cO(\pi))$
  (see~\cite[\S 4.1]{topzeta}) and so $\nu_K(f(\xx)) = \bil\alpha\omega$.
  By~(\ref{p:absint2}), $\bil\alpha\omega \ge 0$ whence
  $\Newton(f)\subset\{\omega\}^*$  since $\bil\alpha\omega \le \bil\beta\omega$ for
  $\beta\in\supp(f)$.
  Thus, $\Newton(f)\subset\bigcap(\{\omega\}^* \!:\!
  \omega\in\cC\cap\ZZ^n) =: \!\Delta$.
  Clearly, $\cC^*\subset \Delta$.
  By rationality, $\cC$ is spanned by certain non-zero
  $\delta_1,\dotsc,\delta_r\in \ZZ^n$ whence $\cC^* =
  \{\delta_1\}^*\cap\dotsb\cap \{\delta_r\}^* \supset \Delta$.
\end{proof}

\subsection{Reminder: computing topological subobject zeta functions}
\label{ss:topzeta2}

We now briefly recall a high-level description of
\cite[Alg.~4.1]{topzeta2} which seeks to compute the topological
zeta function associated with a ``toric datum'' \cite[Def.~3.1]{topzeta2}
$\cT^0$ and a substitution matrix $\beta$.
Toric data are similar to the representation data considered here in that they
contain convex-geometric and algebraic ingredients and that they give rise to
associated $p$-adic integrals and topological zeta functions.
Our method for computing topological representation zeta functions (under
suitable non-degeneracy conditions) in \S\ref{ss:algorithm} is built around the
same core ingredients as \cite[Alg.~4.1]{topzeta2}.

At the heart of \cite[Alg.~4.1]{topzeta2} lies a loop devoted to manipulating two collections of toric
data: unprocessed ones which need further 
investigation and ``regular'' \cite[\S 5.3]{topzeta2} ones whose topological 
zeta functions can be computed explicitly using \cite{topzeta}.
At first, $\cT^0$ is the sole member of the unprocessed collection and the
regular one is empty.

The aforementioned loop proceeds by successively removing a toric datum, $\cT$ say, from
the unprocessed collection.
If $\cT$ fails to be ``balanced'' \cite[\S 5.2]{topzeta2}, a finite collection of
derived balanced toric data is constructed and added back to the unprocessed
collection, whereupon execution of the loop starts again.
If $\cT$ is balanced but not regular, then a
``reduction step''~\cite[\S 7.3]{topzeta2} attempts to construct a finite
collection of toric data derived from $\cT$ which are then added to the
unprocessed collection, just as in the balancing step---it is at this point that
the procedure is allowed to fail.
Finally, if $\cT$ is found to be regular, it is simply added to the regular
collection. 
An intermediate step is concerned with ``simplifications'' \cite[\S
7.2]{topzeta2} of toric data which do not change associated topological
zeta functions.
The aforementioned process of removing elements from the unprocessed collection
and treating them as just described is repeated until we either run out of
unprocessed elements or failure is indicated by the reduction step.
In the case of successful conclusion of this loop,
a finite collection of regular toric data has been constructed.
The topological zeta function associated with each one of these (w.r.t.~the
matrix $\beta$) can be explicitly computed \cite[\S 6.7]{topzeta2}.
By taking the sum of all these rational functions, we finally recover the
topological zeta function associated with $\cT^0$ and $\beta$.

\subsection{A method for computing topological representation zeta functions}
\label{ss:algorithm}

Given a unipotent $k$-group $\GG$,
let $\cR$ be an associated representation datum (Definition~\ref{d:datum}) as in
Theorem~\ref{thm:repint}.
In this subsection, by drawing upon \cites{topzeta,topzeta2}
(see~\S\ref{ss:topzeta2}), we describe a method for explicitly computing 
the topological zeta function $\Zeta^{\cR}_{\topo}(\bm s)$ 
associated with $\cR$ under various restrictions.
As explained in Example~\ref{ex:reps_via_repdatum}, we may then read off
the topological representation zeta function $\zeta_{\GG,\topo}(s)$ of $\GG$.

\subsubsection{Overview of the method}
\label{sss:overview}

\paragraph{Special case: global non-degeneracy.}
We begin by sketching an important special case of our method.
Given an initial representation datum $\cR = (\cD,F,\alpha,\Lambda)$
arising from a unipotent $k$-group $\GG$ as above, 
suppose for the moment that $F$ is globally non-degenerate in the
sense of \cite[Def.~4.2]{topzeta}.
In order to compute $\zeta_{\GG,\topo}(s)$, we then proceed as follows.
First, partition $\cD = \bigcup_{\tau} \cD_\tau$, where $\tau$ ranges over
the faces of the Newton polytope $\cN$ of $\prod F$ and $\cD_\tau = \cD\cap
\NormalCone_\tau(\cN)$ (where $\NormalCone_\tau(\cN)$ denotes the normal cone
of $\tau \subset \cN$).
Let $\cR_\tau := (\cD_\tau,F,\alpha,\Lambda)$. 
Then each $\Zeta^{\cR_\tau}_{\topo}(\bm s)$ can be effectively computed 
via Theorem~\ref{thm:topregular} below---in particular, thanks to global
non-degeneracy of $F$, the Euler characteristics $\Euler(U_G(\CC))$ 
in Theorem~\ref{thm:topregular} can be expressed in terms of convex-geometric
invariants via the Bernstein-Khovanskii-Kushnirenko theorem as in \cite[\S\S 6.1--6.2]{topzeta}. 
Finally, we compute $\zeta_{\GG,\topo}(s) = 1 + \sum_\tau
\Zeta^{\cR_\tau}_{\topo}(a_\lambda s - b_\lambda)_{\lambda\in\Lambda}$.

In contrast to the computation of topological subobject zeta
functions, the above method for cases of global non-degeneracy is already
remarkably useful on its own. However, in order to compute all of the examples given in \S\ref{s:examples},
the more involved approach explained in the following is required.

\paragraph{General case.}
In order to obtain a more powerful method which is applicable at least in some
degenerate cases, we employ a more refined strategy.
The high-level structure of our method coincides
with~\cite[Alg.~4.1]{topzeta2} as summarised in \S\ref{ss:topzeta2}: 
the role of the initial toric datum $\cT^0$ in \cite{topzeta2} is taken by
an initial representation datum (obtained via Theorem~\ref{thm:repint} in the
cases of interest to us), and the family $(a_\lambda,b_\lambda)_{\lambda\in\Lambda}$ corresponds to the substitution
matrix $\beta$ from \cite{topzeta2}. 
Keeping in mind Example~\ref{ex:reps_via_repdatum}, it remains

\begin{enumerate}[(a)]
\item
  to define the notions of balanced and regular
  representation data
  and to provide methods for the tasks of
  balancing and regularity testing (\S\ref{sss:balreg}),
\item
  to describe the topological evaluation of regular representation data
  (\S\ref{sss:topeval}), and 
\item
  to explain the simplification and
  reduction steps (\S\ref{sss:simred}).
\end{enumerate}

In the following, we concisely address these points while pointing to related
sections of \cites{topzeta,topzeta2}. 
With all these ingredients in place, 
while the resulting method is still allowed to fail in the reduction step,
in the same way that \cite{topzeta2} extended the scope of \cite{topzeta},
it is more powerful than the method sketched in the ``special case'' above which
relied solely on non-degeneracy assumptions.

\subsubsection{Balanced and regular representation data}
\label{sss:balreg}

Let $f\in k[\XX^{\pm 1}]$, say $f =
\sum_{\alpha\in\ZZ^n}c_\alpha \XX^\alpha$ ($c_\alpha\in k$).
Let $\omega\in \RR^n$. 
If $f\not= 0$, recall that the \emph{initial form}
$\init_\omega(f)$ is the sum of those $c_\alpha
\XX^\alpha$ for $\alpha\in \supp(f) = \{ \gamma\in \ZZ^n:c_\gamma\not= 0\}$
where $\bil\alpha\omega$ is minimal;
we let $\init_\omega(0) = 0$.

The following is a natural substitute for the balanced and regular toric
data in \cite{topzeta2}.
Henceforth, let $\bar k$ be a fixed algebraic closure of $k$.

\begin{defn}[{Cf.~\cite[Def.~5.1, 5.5]{topzeta}}]
  \label{d:balreg}
  A representation datum $(\cD,F,\alpha,\Lambda)$ with $\cD\not= \emptyset$ is
  \emph{balanced} if for each $f\in F$, the initial form
  $\init_\omega(f)$ remains constant for $\omega \in \cD$.
  It is \emph{regular} if, in addition, 
  for all $G\subset F$, if $\xx\in \Torus^n(\bar k)$ satisfies
  $\init_\omega(g)(\xx) = 0$ for all $g\in G$, then the rank of
  $\left[ \frac{\partial\!\init_\omega(g)}{\partial X_i}(\xx)\right]_{g\in G,i=1,\dotsc,n}$ is
  $\noof G$, where $\omega\in \cD$ is arbitrary.
\end{defn}
The balancing procedure and regularity testing~\cite[pp.~12--13]{topzeta2} carry
over readily.

\subsubsection{Topological evaluation}
\label{sss:topeval}

The following theorem allows us to explicitly compute the topological zeta functions
associated with regular representation data satisfying the given additional assumptions.

\begin{thm}
  \label{thm:topregular}
  Let $\cR = (\cD,F,\alpha,\Lambda)$ be a regular representation datum in $n$ variables.
  Suppose that $\cD\subset\partial\Orth^n$ and that each element of $F$ is
  homogeneous.
  Then there are \textbf{explicit} $k$-varieties $U_G$ and rational functions $W_G\in
  \MM[X,(Y_\lambda)_{\lambda\in\Lambda}]$ indexed by subsets $G\subset F$ and
  a finite set $S\subset \Places_k$ such
  that for all $v\in \Places_k\setminus S$ and all finite extensions $K/k_v$,
  \[
  (1-q_K^{-1})^{-1} \Zeta^\cR_K(\bm s) =
  \sum_{G\subset F} \noof{\bar U_G(\fO_K/\fP_K)}\dtimes
  W_G\!\left(q_K^{\phantom{-s_1}}\!\!\!\!,\left(q_K^{-s_\lambda}\right)_{\lambda\in\Lambda}\right)\!,
  \]
  where $\bar\dtimes$ denotes reduction modulo $\fp_v$ of fixed but arbitrary
  $\fo$-models of the $U_G$.
  Hence, $\cR$ is $(-1$)-expandable 
  and $\Zeta^\cR_{\topo}(\bm s) = \sum\limits_{G\subset
    F}\Euler(U_G(\CC))\dtimes \red{W_G}$.
\end{thm}

\begin{rem}
  The assumptions on $\cD$ and $F$ in the second sentence of Theorem~\ref{thm:topregular}
  are satisfied if $\cR$ is obtained via Theorem~\ref{thm:repint}.
  Moreover, the simplification and reduction steps below are designed to
  preserve these properties.
  Finally, the validity of the univariate substitutions from Theorem~\ref{thm:repint}
  in the sense of \cite[Rem.~5.15]{topzeta} is also preserved at all times.
  Hence, starting with $\cR$ from Theorem~\ref{thm:repint},
  after successful termination of the procedure explained in \S\ref{sss:overview}, we have indeed
  computed $\zeta_{\GG,\topo}(s)-1 = \Zeta^{\cR}_{\topo}(a_\lambda
  s-b_\lambda)_{\lambda\in\Lambda}$.
\end{rem}

\begin{proof}[Proof of Theorem~\ref{thm:topregular}]
We begin by stating an analogue for representation data of the explicit
$p$-adic formula in \cite[Thm~5.8]{topzeta2}.
Let $\cR =(\cD,F,\alpha,\Lambda)$ be any regular representation datum with $\cD\not=\emptyset$ and
let $\omega\in\cD$ be arbitrary.
For $f\in F$, choose $\phi(f)\in \supp(\init_\omega(f))$.
Write $\Torus^d_k := \Torus^d \otimes_{\ZZ} k$.
For $G\subset F$, let $V_G$ be the subvariety of $\Torus^n_k$ defined by the
vanishing of $\init_\omega(g)$ for $g\in G$ and the non-vanishing of
$\prod_{f\in F\setminus G}\init_\omega(f)$.
Let $(\ee_g)_{g\in G}$ be the standard basis of $\RR^G$.
For $\lambda =(i,j)\in \Lambda$, define a lattice polytope
\[
\cP_{G,\lambda} = \conv\!\left(
  \begin{aligned}
    &(\phi(g) + \alpha_i(g),0,\ee_g), & \quad (g\in F_i\cap G)\\
    &(\phi(f) + \alpha_i(f),0,0), & \quad (f\in F_i\setminus G,f\not=0) \\
    &(\phi(g) + \alpha_j(g),1,\ee_g), & \quad (g\in F_j\cap G) \\
    &(\phi(f) + \alpha_j(f),1,0) & \quad (f\in F_j\setminus G,f\not= 0)
    \end{aligned}
  \right) \subset \cD^* \times\Orth^{\phantom G}\times\Orth^G.
\]

As a minor variation of \cite[\S 3.3]{topzeta},
given a finite union of rational half-open cones $\cH\subset\Orth^l$ and a
finite collection $(\cQ_\iota)_{\iota\in I}$
of lattice polytopes $\cQ_\iota\subset \cH^*$, we obtain a unique 
$\cZ^{\cH,(\cQ_\iota)_{\iota\in I}} \in \QQ(X,(Y_\iota)_{\iota\in I})$
characterised by the identity of formal power series
\[
\cZ^{\cH,(\cQ_\iota)_{\iota\in I}} =
\sum_{\omega\in\cH\cap\ZZ^l} X^{-\bil{(1,\dotsc,1)}\omega} \prod_{\iota\in I}
Y_\iota^{\min(\bil\psi\omega : \psi\in \cQ_\iota)}.
\]
A straight-forward modification of \cite[Thm~4.10]{topzeta} now yields
that for almost all $v\in \Places_k$ and all finite extension $K/k_v$,
we have
\[
\Zeta^\cR_K(\bm s) = 
\sum_{G\subset F} 
\noof{\bar V_G(\fO_K/\fP_K)} \dtimes \frac{(q_K-1)^{\noof G+1}}{q_K^{n+1}} \dtimes \cZ^{
  \cD\times\StrictOrth^{\phantom G}\times\StrictOrth^G, (\cP_{G,\lambda})_{\lambda\in\Lambda}}(q_K^{\phantom{s_1}},(q_K^{-s_\lambda})_{\lambda\in\Lambda}).
\]
Now assume that $\cD\subset\partial\Orth^n$ and that $F$ consists of homogeneous elements.
Let $\cN$ denote the Newton polytope of $\prod F$.
Since $\cR$ is balanced, there exists a unique face
$\tau\subset\cN$ such that $\cD$ is contained in the  normal cone
$\NormalCone_\tau(\cN)$ of $\cN$ with respect to $\tau$,
cf.~\cite[Lem.~5.3]{topzeta2}.
By~\cite[Lem.~6.1(iii)]{topzeta}, $\tau$ coincides with the Newton
polytope of $\prod_{f\in F}\init_{\omega}(f)$.
Using~\cite[Lem.~6.1(i)]{topzeta}, we construct
varieties $U_G\subset \Torus^{\dim(\tau)}_k$ and explicit isomorphisms
$V_G \approx_k U_G \times_k \Torus^{n-\dim(\tau)}_k$.
Hence, the formula for  $(1-q_K^{-1})^{-1} \Zeta^{\cR}_K(\bm s)$ in
the statement of the current theorem holds for
\[
W_G :=
X^{-n} (X-1)^{\noof G + n - \dim(\tau)} \dtimes
\cZ^{\cD\times\StrictOrth^{\phantom G}\times\StrictOrth^G, (\cP_{G,\lambda})_{\lambda\in\Lambda}}(X,(Y_\lambda)_{\lambda\in\Lambda})
\]
and it only remains to prove that $W_G\in \MM[X,(Y_\lambda)_{\lambda\in\Lambda}]$.
Write $\cD = \bigcup_{\delta\in\Delta}\cC_0^\delta$ as in
Definition~\ref{d:datum} for non-empty $\cC_0^\delta$.
By~\cite[Lem.~6.11]{topzeta2} and since
$\dim\bigl(\cC_0^\delta\times\StrictOrth\times\StrictOrth^G\bigr) = \dim\bigl(\cC_0^\delta\bigr) + \noof
G + 1$, it suffices to show that $\dim(\cC_0^\delta) < n-\dim(\tau)$ for $\delta\in\Delta$.
Fix $\delta\in \Delta$ and let $\bm w_1,\dotsc,\bm w_e\in \cC_0^\delta$ be
$\RR$-linearly independent for $e = \dim\bigl(\cC_0^\delta\bigr)$.
As $F$ consists of homogeneous elements, 
the vector $(1,\dotsc,1)$ is contained in the closure of every normal cone of
$\cN$.
In particular, there exists 
$\bm w_{e+1} \in \NormalCone_\tau(\cN) \cap \StrictOrth^n$.
By convexity, $\cC_0^\delta\subset\partial\Orth^n$ is contained
in a single coordinate hyperplane of $\RR^n$.
We conclude that $\{\bm w_1,\dotsc,\bm w_{e+1}\}\subset\NormalCone_\tau(\cN)$ is
linearly independent so that $\dim(\cC_0^\delta) = e<
\dim(\NormalCone_\tau(\cN)) = n-\dim(\tau)$.
\end{proof}

In order to explicitly compute the Euler characteristics $\Euler(U_G(\CC))$ and
rational functions $\red{W_G}$ in Theorem~\ref{thm:topregular}, we proceed similarly to
\cite[\S 3.3]{topzeta} and \cite[\S\S 6.5--6.7]{topzeta2}.

\subsubsection{Simplification and reduction}
\label{sss:simred}

\paragraph{Change of generators.}
Let $\cR=(\cD,F,\alpha,\Lambda)$ with $F = F_1\cup\dotsb\cup
F_m$ as in Definition~\ref{d:datum}.
Let $\cR'=(\cD',F',\alpha',\Lambda')$ with $F'= F_1'\cup\dotsb\cup F_m'$ be
another representation datum in the same number of variables as $\cR$.
Suppose that $\cD = \cD'$, $\Lambda = \Lambda'$,
and that $\cR$ and $\cR'$ are both $(-1)$-expandable.
If $\XX^{\alpha_i^{\phantom\prime}} F_i^{\phantom\prime}$ and $\XX^{\alpha_i'} F_i^\prime$ generate the same ideal of
$k[\cD^*\cap\ZZ^n]$ for $i=1,\dotsc,m$, then $\Zeta^\cR_{\topo}(\bm s) =
\Zeta^{\cR'}_{\topo}(\bm s)$.
This simple observation lies at the heart of the 
simplification and reduction steps explained in the following.
It will be most convenient to describe these operations on the level of the sets $\hat F_i := \XX^{\alpha_i} F_i$.
Specifically, beginning with $\hat F_1,\dotsc, \hat F_m$, we will derive sets
$\hat F_1', \dotsc,\hat F_m'$.
We then obtain $\cR'$ by constructing $F' \subset k[\XX^{\pm 1}]$ of
minimal cardinality, $F_1',\dotsc,F_m'\subset F'$, and $\alpha_i'\colon
F_i'\to\ZZ^n$  with $\hat F_i' = \XX^{\alpha_i'} F_i'$ for $i=1,\dotsc,m$.
While $F'$ is only determined by $(\hat F_1',\dotsc,\hat F_m')$ up to
rescaling by Laurent monomials, this ambiguity not does affect
questions of degeneracy, see~\cite[Rem.~4.3(ii)]{topzeta}. 

\paragraph{Simplification.}
Let $\cR = (\cD,F,\alpha,\Lambda)$ and $\hat F_i = \XX^{\alpha_i}F_i$ be as
above.
We obtain $\hat F_i'$ from $\hat F_i$ by repeated application of the
following steps.
First, if $f$ divides $g$ within $k[\cD^*\cap\ZZ^n]$ for distinct $f,g\in \hat
F_i$, then we discard $g$ from $\hat F_i$.
Next, if $f,f'\in \hat F_i$ are distinct, if $t,t'$ are terms
(of initial forms, if $\cR$ is balanced) of $f$ and $f'$,
respectively, and if $t/t'\in k[\cD^*\cap\ZZ^n]$, then we are free to replace
$f$ by $g := f - f't/t'$ in $\hat F_i$ which we do if
$\card{\supp(g)}<\card{\supp(f)}$. 
After finitely many iterations, this procedure stops at which point we have
constructed sets $\hat F_1',\dotsc,\hat F_m'\subset k[\cD^*\cap\ZZ^n]$.
We then construct $\cR'$ as previously explained.
We note that in contrast to the simplification step for toric data
in \cite[\S 7.2]{topzeta2}, in the present setting, we cannot simply discard
monomials from $F$ by modifying $\cD$.

\paragraph{Reduction.}
As in \cite[\S 7.3]{topzeta2}, the reduction step is a last resort (which might
well fail) which is only ever attempted when a balanced and simplified
representation datum $\cR = (\cD,F,\alpha,\Lambda)$ is singular (i.e.~fails to be regular).
Again, write $\hat F_i = \XX^{\alpha_i} F_i$.
We first construct a $\subset$-minimal set $G\subset F$ witnessing the failure of
regularity in Definition~\ref{d:balreg}.
We assume that $\card{F_i\cap G}\ge 2$ for some $i$; if no such $i$
exists, then we give up right away.
Define $\hat\alpha_i\colon F_i\to \hat F_i, f \mapsto \XX^{\alpha_i(f)} f$ 
and note that  $\hat\alpha_i$ is injective since we already applied the
simplification procedure.
We may therefore find distinct $f,f'\in \hat\alpha_i(F_i\cap G)$.
We next apply an algebraic transformation to each of $f$ and
$f'$ in turn, resulting in two derived representation data $\cR^+$ and
$\cR^-$. Our ultimate goal is to remove the particular cause of singularity of $\cR$
corresponding to $(f, f')$.
Specifically, given $(f,f')$, analogously to \cite[\S 7.3]{topzeta2}, we next
consider pairs $(t,t')$ consisting of terms of $\init_\omega(f)$ and
$\init_\omega(f')$, respectively, where $\omega\in\cD$.
Having chosen $(f,f')$ and $(t,t')$ (cf.~\cite[\S 7.3]{topzeta2}),
let $\gamma\in\ZZ^n$ be the exponent vector of the unique monomial in $t/t'$.
Let $\cD^\pm = \cD \cap \{\pm \gamma\}^*$
and define $\hat F_i^{\pm}$ by replacing $f$ (resp.~$f'$) by $f - \frac{t}{t'} f'$
(resp.~$f' - \frac{t'}{t} f$) within $\hat F_i$; we let $\hat F_j^\pm = \hat
F_j^{\phantom \pm}\!$
for $j\not= i$.
By proceeding as indicated above, we obtain
representation data $\cR^{\pm} = (\cD^\pm,F^{\pm},\alpha^{\pm},\Lambda)$ and
by construction, we have $\Zeta^{\cR}_{\topo}(\ess) =
\Zeta^{\cR^+}_{\topo}(\ess) + \Zeta^{\cR^-}_{\topo}(\ess)$
(assuming $(-1)$-expandability).
We then add $\cR^+$ and $\cR^-$ back to the unprocessed collection of our main
procedure (see~\S\ref{ss:topzeta2}) and resume with the next iteration of its
main loop.
As in \cite[\S 7.3]{topzeta2},
there is no guarantee that the above operations will ever
succeed in removing all singularities.
In order to ensure termination, we again impose a bound on the
total ``reduction depth''.

\section{A full classification in dimension six and further examples}
\label{s:examples}

We now illustrate the strength of our method for computing topological
representation zeta functions of unipotent groups from \S\ref{s:compute}
by presenting a large number of examples computed with its help.
In particular, we give a complete list of the topological
representation zeta functions associated with unipotent groups of dimension at
most six over an algebraic closure $\bar\QQ$ of $\QQ$.
The success of our approach in these cases is in notable contrast to the
computation of topological subgroup zeta functions in 
\cites{topzeta,topzeta2} where substantial gaps remain in dimension six. 

\subsection*{Practical matters}
In order to explicitly compute topological representation zeta functions of
unipotent groups using the method explained in \S\ref{s:compute}, computer
assistance is indispensable for all but the smallest examples.
The author's publicly available package \textsf{Zeta}~\cite{zeta2} for the computer algebra system
Sage~\cite{Sage} provides an implementation of the method from \S\ref{s:compute} for unipotent
$\QQ$-groups; all the examples discussed below were computed using it.
We included references to associated $p$-adic
representation zeta functions whenever we are aware of them.
In all these cases, our machine computations agree with the rational functions
obtained using the informal method for deducing topological zeta functions from
$p$-adic ones from the introduction.
Our automated approach therefore provides some independent confirmation of these
manual computations.
The majority of examples below are new in the sense that we are not aware of any
previous computations of associated $p$-adic representation zeta functions.

\subsection*{Classification by dimension}
Consider the problem of determining the topological
representation zeta functions of unipotent $k$-groups of dimension at most some
fixed number $d$ as $k$ ranges over number fields.
By invariance under base extension (Proposition~\ref{prop:base}) and since
unipotent groups in characteristic zero are in 1-1 correspondence
with finite-dimensional nilpotent Lie algebras, the topological zeta functions
in question are naturally indexed by nilpotent Lie algebras of dimension at most
$d$ over $\bar\QQ$.
By multiplicativity (Corollary~\ref{cor:product}), we may
further restrict attention to $\oplus$-indecomposable algebras.

\subsection*{A complete list in dimension at most six}
A classification of $6$-dimensional nilpotent Lie algebras over a field $F$ of
characteristic zero was first given in \cite{Mor58}.
The number of isomorphism classes of these algebras is finite if and only if
$F^\times/(F^\times)^2$ is.
In particular, the task of determining the topological representation zeta
functions of unipotent $\bar\QQ$-groups of dimension at most six is a finite
problem.
It turns out that all groups in questions are amenable to the method
from~\S\ref{s:compute}.
A complete list of the topological representation zeta functions
corresponding to the $30$~indecomposable nilpotent Lie algebras of dimension at
most six over $\bar\QQ$
is given in Table~\ref{tab:six} (p.~\pageref{tab:six}).
In the column ``$\bm\fg$'', we give the names
of these algebras as in~\cite[\S 4]{dG07}.
For half of the Lie algebras in Table~\ref{tab:six}, associated $p$-adic
representation zeta functions have been previously computed.
In cases where the author is aware of such computations, the column ``$p$-adic
reference'' in Table~\ref{tab:six} points to these $p$-adic results.

The column ``wt''  provides a measure of the ``algebro-geometric complexity''
of the computation of the respective topological zeta function.
In detail, by the \emph{weight} of a representation datum
$(\cD,F,\alpha,\Lambda)$, we mean the number $\sum_{f\in F}
(\card{\supp(f)}-1)\ge 0$.
In particular, the weight of $(\cD,F,\alpha,\Lambda)$ is zero precisely when $F$
consists entirely of Laurent monomials.
The weights given in Table~\ref{tab:six} are those
of the initial representation data constructed using Theorem~\ref{thm:repint}
(w.r.t.~a suitable choice of a basis of $\bm\fg$).

In the case of weight~0, the computation of topological representation
zeta functions via Theorem~\ref{thm:topregular} is purely combinatorial
and immediately reduces to computing the topological incarnation
of a single rational function $\cZ^{\cH,(\cQ_\iota)_{\iota\in I}}$ from the
proof of Theorem~\ref{thm:repint}.
While this case might be devoid of algebraic geometry, it can easily be
computationally expensive due to the large number of case distinctions
and subdivisions of cones involved.

\begin{table}[h]
  \footnotesize
  \centering
  \begin{tabular}{r|c|l|c|r}
   \hline
   $\bm\fg$ & dim & $\zeta_{\GG,\topo}(s)$ & wt & $p$-adic reference
   \\
    \hline
    abelian & 1 & $1$ & 0 & obvious\\
    $L_{3,2}$ & 3 & $\frac s {s-1}$ & 0 & \cite[Thm~5]{NM89} \\
    $L_{4,3}$& 4 & $\frac{s^2}{(s - 1)^2}$ & 0 & $M_3$ \cite[(4.2.24)]{Ezzat}\\
    $L_{5,4}$& 5 & $\frac{2s}{2s - 1}$ & 0 & $B_4$ \cite[Ex.~6.3]{Snocken} \\
    $L_{5,5}$& 5 & $\frac{(2s - 1)s}{2(s - 1)^2}$ & 0 & $G_{5,3}$ \cite[Tab.~5.2]{Ezzat}\\
    $L_{5,6}$& 5 & $\frac{2s^2}{(2s - 3)(s - 1)}$ & 0\\
    $L_{5,7}$& 5 & $\frac{s^2}{(s - 1)(s - 2)}$ & 0 & $M_4$ \cite[(4.2.24)]{Ezzat}\\
    $L_{5,8}$& 5 & $\frac{s}{s - 2}$ & 0 & $M_{3,3}$ \cite[(5.3.7)]{Ezzat}; $G_3$ \cite[Ex.~6.2]{Snocken}\\
    $L_{5,9}$& 5 & $\frac{s^2}{(s - 1)(s - 2)}$ & 0 & $F_{3,2}$ \cite[Tab.~5.2]{Ezzat}\\
    $L_{6,10}$& 6 & $\frac{2s^2}{(2s - 1)(s - 1)}$ & 0 & $G_{6,12}$ \cite[Tab.~5.2]{Ezzat}\\
    $L_{6,11}$& 6 & $\frac{(6s + 1)s}{2(3s - 4)(s - 1)}$ & 0 \\
    $L_{6,12}$& 6 & $\frac{2s^2}{(2s - 3)(s - 1)}$ & 0 \\
    $L_{6,13}$& 6 & $\frac{(12s^2 - 18s + 7)s^2}{6(2s - 1)(s - 1)^3}$ & 1 \\
    $L_{6,14}$& 6 & $\frac{(12s^2 - 12s + 1)s^2}{3(2s - 1)(2s - 3)(s - 1)^2}$ & 1 \\
    $L_{6,15}$& 6 & $\frac{(6s - 7)s^2}{(3s - 5)(2s - 3)(s - 1)}$ & 1 \\
    $L_{6,16}$& 6 & $\frac{2s^3}{(2s - 1)(s - 1)(s - 2)}$ & 0 \\
    $L_{6,17}$& 6 & $\frac{(2s - 3)s^2}{2(s - 1)(s - 2)^2}$ & 0 \\
    $L_{6,18}$& 6 & $\frac{s^2}{(s - 1)(s - 3)}$ & 0 & $M_5$ \cite[(4.2.24)]{Ezzat}\\
    $L_{6,19}(0)$ & 6 & $\frac{s^2}{(s - 1)(s - 2)}$ & 0 & $G_{6,7}$ \cite[Tab.~5.2]{Ezzat}\\
    $L_{6,19}(1)$ & 6 & $\frac{2s^2}{(2s - 1)(s - 2)}$ & 0 & $G_{6,14}$ \cite[Tab.~5.2]{Ezzat}\\
    $L_{6,20}$ & 6 & $\frac{(2s - 1)s}{2(s - 1)(s - 2)}$ & 0 \\
    $L_{6,21}(0)$ & 6 & $\frac{s^2}{(s - 2)^2}$ & 0 \\
    $L_{6,21}(1)$ & 6 & $\frac{2s^2}{(2s - 3)(s - 2)}$ & 0 \\
    $L_{6,22}(0)$ & 6 & $\frac{2s}{2s - 3}$ & 0 & \cite[Ex.~6.5]{Snocken}\\
    $L_{6,22}(1)$ & 6 & $\frac{s^2}{(s - 1)^2}$ & 1\\
    $L_{6,23}$& 6 & $\frac{(2s - 3)s}{2(s - 2)^2}$ & 0 \\
    $L_{6,24}(0)$& 6 &$\frac{(4s^2 - 6s + 1)s}{(2s - 3)^2(s - 1)}$ & 0 \\
    $L_{6,24}(1)$& 6 &$\frac{(2s + 1)s}{(2s - 3)(s - 1)}$ & 1 \\
    $L_{6,25}$& 6 & $\frac{(s - 1)s}{(s - 2)^2}$ & 0 & $M_{4,3}$ \cite[(5.3.7)]{Ezzat}\\
    $L_{6,26}$& 6 &$\frac s{s - 3}$ & 0 & $F_{1,1}$ \cite[Thm~B]{SV14}\\
    \hline
  \end{tabular}
  \caption{Topological representation zeta functions 
    associated with $\oplus$-indecomposable nilpotent Lie algebras of
    dimension at most six (complete list)}
  \label{tab:six}
\end{table}

\subsection*{Further examples}
To the author's knowledge,
seven is the largest dimension for which a complete classification
of nilpotent Lie $\CC$-algebras is known; see
\cite{Mag10} for a recent comparison, including some corrections, of such classifications.
Beginning in dimension seven, infinite families of pairwise non-isomorphic nilpotent Lie
algebras appear over $\CC$ and over $\bar \QQ$ whence the above reduction to a
finite computation no longer applies. 
In addition, there are $7$-dimensional nilpotent Lie $\bar\QQ$-algebras which
are not amenable to our method.
Despite these limitations, we can compute a large number of
interesting examples of topological representation zeta functions of unipotent
$\bar\QQ$-groups of dimension seven and beyond,
and Table~\ref{tab:beyond} (p.~\pageref{tab:beyond}) includes a selection of these.
To the author's knowledge, no $p$-adic representation zeta functions are known
for any of the examples in Table~\ref{tab:beyond}.
The first batch of Lie algebras in Table~\ref{tab:beyond} is taken from
\cite{See93}.
The algebras $N_i^{8,d}$ are taken from the lists of $8$-dimensional
Lie $\CC$-algebras of class $2$ with $d$-dimensional
centre in \cites{RZ11,YD13}.
The remaining algebras in Table~\ref{tab:beyond} are obtained from algebras in
Table~\ref{tab:six} by base 
change to dual numbers. In detail, let $k[\varepsilon] = k[X]/X^2$
and for a $k$-algebra $\bm\fg$, let $\bm\fg[\varepsilon] = \bm\fg\otimes_k
k[\varepsilon]$ regarded as a $k$-algebra.
For example, $L_{3,1}$ is the Heisenberg Lie algebra and $L_{3,1}[\varepsilon]
\approx L_{6,22}(0)$.
Recall from Propositions~\ref{prop:base} and~\ref{prop:res} that base extension
of number fields followed by restriction of scalars is very well-behaved
on the level of topological representation zeta functions.
The examples in Table~\ref{tab:beyond} show that this is not generally true for
the operation $\bm\fg\mapsto \bm\fg[\varepsilon]$  (but see
Question~\ref{qu:3over2} below).
For yet more examples, we refer the reader to the database of topological
representation zeta functions included with \cite{zeta2}.

\begin{table}[h]
  \footnotesize
  \centering
  \begin{tabular}{r|c|l|c}
   \hline
   $\bm\fg$ & dim &  $\zeta_{\GG,\topo}(s)$ & wt
   \\
   \hline
   $3,5,7_C$   \cite[p.~483]{See93} & 7 & $\frac{(2s - 3)(s - 1)s}{2(s - 2)^3}$ & 1 \\
   $2,7_B$ \cite[p.~484]{See93} & 7 & $\frac s {s-1}$ & 0 \\
   $2,5,7_D$ \cite[p.~484]{See93} & 7 & $\frac{4(s - 1)s}{(2s - 3)^2}$ & 0 \\
   $2,5,7_G$ \cite[p.~484]{See93} & 7 & $\frac{(12s^2 - 22s + 9)s}{2(3s - 4)(2s - 3)(s - 1)}$ & 1 \\
   $2,4,7_J$ \cite[p.~485]{See93} & 7 & $\frac{(6s + 1)s}{2(3s - 4)(s - 2)}$ & 1 \\
   $2,4,7_R$ \cite[p.~486]{See93} & 7 & $\frac{(6s^2 - 14s + 7)s}{2(3s - 4)(s - 2)^2}$ & 2 \\
   $2,4,5,7_K$ \cite[p.~487]{See93} & 7 & $\frac{2(24s^3 - 82s^2 + 84s - 23)s}{(4s - 7)(3s - 5)(2s - 3)^2}$ & 1 \\
   $2,3,5,7_A$ \cite[p.~488]{See93} & 7 & $\frac{2s^2}{(2s - 3)(s - 2)}$ & 3 \\
   $2,3,5,7_D$ \cite[p.~488]{See93} & 7 & $\frac{(36s^2 - 24s - 29)s}{2(3s - 4)^2(2s - 3)}$ & 2 \\
   $2,3,4,5,7_D$ \cite[p.~489]{See93} & 7 & $\frac{(6s^3 - 22s^2 + 25s - 8)s^2}{2(3s - 5)(s - 1)^2(s - 2)^2}$ & 2 \\
   $2,3,4,5,7_E$ \cite[p.~489]{See93} & 7 & $\frac{2s^2}{(2s - 5)(s - 2)}$ & 3 \\
   $1,5,7$ \cite[p.~489]{See93} & 7 & $\frac{(3s - 1)s}{(3s - 2)(s - 1)}$ & 0 \\
   $1,4,7_C$ \cite[p.~489]{See93} & 7 & $\frac{3(s - 1)s}{(3s - 4)(s - 3)}$ & 0 \\
   $1,4,5,7_B$ \cite[p.~490]{See93} & 7 & $\frac{(3s - 2)s^2}{3(s - 1)^3}$ & 0 \\
   $1,3,7_B$  \cite[p.~490]{See93} & 7 & $\frac{(9s^3 - 9s^2 + 3s - 1)s}{(3s - 2)^2(s - 1)^2}$ & 0 \\
   $1,3,5,7_M:0$ \cite[p.~491]{See93} & 7 & $\frac{(12s^2 - 42s + 37)(s - 1)s}{6(2s - 3)(s - 2)^3}$ & 1 \\
   $1,3,5,7_S:1$ \cite[p.~491]{See93} & 7 & $\frac{(36s^2 - 24s - 29)s}{2(3s - 4)^2(2s - 3)}$ & 2 \\
   $1,3,4,5,7_I$   \cite[p.~492]{See93}&7 & $\frac{(3s - 4)s^2}{(3s - 5)(s - 2)^2}$ & 1 \\
   $1,2,4,5,7_E$  \cite[p.~493]{See93}& 7 & $\frac{(288 s^3 - 624 s^2 + 266 s + 81)s}{(6 s - 7) (4 s - 5)^2 (3s - 4)}$ & 2 \\
   $1,2,4,5,7_J$ \cite[p.~493]{See93} & 7 & $\frac{2(6s - 7)s^2}{3(2s - 3)^2(s - 2)}$ & 2 \\
   $1,2,3,5,7_B$ \cite[p.~494]{See93} & 7 & $\frac{(8s^2 - 12s + 5)s^2}{8(s - 1)^4}$ & 3 \\
   $1,2,3,5,7_C$ \cite[p.~494]{See93} & 7 & 
   $\frac{(3024 s^5 - 10248 s^4 + 13286 s^3 - 7893 s^2 + 1900 s - 64)s}{56(9s-8)(6s-5)(s-1)^4}$
   & 3 \\
   $1,2,3,4,5,7_F$ \cite[p.~494]{See93} & 7 & $\frac{(3s - 4)(2s - 3)s^2}{2(3s- 5)(s - 1)(s - 2)^2}$ & 3 \\
   $1,2,3,4,5,7_G$ \cite[p.~494]{See93} & 7 & $\frac{(4s - 5)s^2}{(4s - 7)(s-1)(s - 2)}$ & 2 \\
   $N_1^{8,2}$ \cite[Thm~1]{RZ11} & 8 & $\frac{6s^2}{(3s - 2)(2s - 1)}$ & 1 \\
   $N_5^{8,2}$ \cite[Thm~1]{RZ11} & 8 & $\frac{6s}{6s - 5}$ & 0 \\
   $N_3^{8,3}$ \cite[Thm~3.2]{YD13} & 8 & $\frac{(4s^2 - 6s + 1)s}{(2s - 3)^2(s -1)}$ & 0 \\
   $N_7^{8,3}$ \cite[Thm~3.2]{YD13} & 8 & $\frac{(6s - 7)s}{(3s - 5)(2s - 3)}$ & 1 \\
   $N_1^{8,4}$ \cite[Thm~3.3]{YD13} & 8 & $\frac{(s - 1)s}{(s - 2)(s - 3)}$ & 1\\
   $N_3^{8,4}$ \cite[Thm~3.3]{YD13} & 8 & $\frac{2(s - 2)s}{(2s - 5)(s - 3)}$ & 1 \\
   $L_{4,3}[\varepsilon]$ & 8 & $\frac{2(4s^2 - 6s + 1)s}{(2s - 3)^3}$ & 1 \\
   $L_{5,4}[\varepsilon]$ & 10 & $\frac{4s}{4s - 3}$ & 0 \\
   $L_{5,5}[\varepsilon]$ & 10 & $\frac{16(s - 1)^2s}{(4s - 5)(2s - 3)^2}$ & 2 \\
   $L_{5,7}[\varepsilon]$, $L_{5,9}[\varepsilon]$ & 10 & $\frac{(2s^2 - 4s + 1)s}{(2s - 5)(s - 2)^2}$ & 3 \\
   $L_{5,8}[\varepsilon]$ & 10 & $\frac{(2s - 3)s}{(2s - 5)(s - 2)}$ & 1 \\
   $L_{6,22}(0)[\varepsilon] \approx L_{3,2}[\varepsilon][\varepsilon']$ & 12 & 
   $\frac{2(2s - 3)s}{(4s - 7)(s - 2)}$  & 2 \\
   $L_{6,23}[\varepsilon]$ & 12 & 
   $\frac{2(8s^3 - 44s^2 + 82s - 53)s}{(4s - 9)(2s - 5)^2(s - 2)}$
   & 5 \\ 
   $L_{6,24}(0)[\varepsilon]$ & 12 & $\frac{2(512 s^6 - 4544 s^5 + 16544 s^4 - 31500 s^3 + 32885 s^2 - 17685 s +  3769)s}{(8 s - 15)(4 s - 7)^3 (2 s - 3)(s - 2)^2}$
   & 5 \\
   $L_{6,25}[\varepsilon]$ & 12 &  $\frac{2(4s^3 - 20s^2 + 33s - 19)s}{(2s - 5)^3(s - 2)}$ & 3 \\
   $L_{6,26}[\varepsilon]$ & 12 & $\frac{2(s - 2)s}{(2s - 7)(s - 3)}$ & 3 \\
   \hline
 \end{tabular}
 \caption{Examples of topological representation zeta functions in dimension $\ge 7$}
 \label{tab:beyond}
\end{table}

\section{Open questions}
\label{s:questions}

Based on our experimental evidence, we state some open questions
which might provide interesting avenues for future research.
Throughout, let $\GG$ be a non-abelian unipotent $k$-group.
All examples of topological representation zeta functions
that we computed are consistent with
Questions~\ref{qu:refined_degree}--\ref{qu:fp_bounds} below having positive
answers;
using the explicit $p$-adic formulae from \cite[Thm~B]{SV14}, this includes
examples in much higher dimensions than those covered by Tables~\ref{tab:six}--\ref{tab:beyond}.

By Corollary~\ref{cor:degree}, $\zeta_{\GG,\topo}(s)$ has
degree zero in $s$.
In contrast, no explanation of the observed degrees seems to be known for 
topological subalgebra zeta functions, see \cite[\S 8]{topzeta};
note that in the enumeration of subalgebras, passing from $p$-adic to
topological zeta functions involves an additional
transformation~\cite[Ex.~5.11(iii), Def.~5.17]{topzeta}.
Perhaps the following question is a more appropriate analogue of
\cite[Conj.~I]{topzeta}. 

\begin{question}
  \label{qu:refined_degree}
  Does $\zeta_{\GG,\topo}(s)-1$ always have degree $-1$ in $s$?
\end{question}

We now consider a refinement of
Question~\ref{qu:refined_degree} in the spirit of
\cite[\S 9.3]{topzeta2}.
Define $$\omega(\GG) := 
s(\zeta_{\GG,\topo}(s)-1)\Big\vert_{s=\infty} =
s^{-1}\bigl(\zeta_{\GG,\topo}(s^{-1})-1\bigr)\Big\vert_{s=0} \in
\QQ.$$
Question~\ref{qu:refined_degree} has a positive answer if and only if
always $\omega(\GG)\not= 0$.
For example, the groups with Lie algebras $1,\!4,\!5,\!7_B$, $1,\!3,\!7_B$, and
$1,\!2,\!3,\!5,\!7_C$ in Table~\ref{tab:beyond} have $\omega$-invariant $\frac 7
3$.

\begin{question}
  Is $\omega(\GG)$ always positive?
\end{question}

It would be interesting to find a group-theoretic interpretation of $\omega(\GG)$.
Using Corollary~\ref{cor:product} and Proposition~\ref{prop:infinity},
it is easy to see that if $\HH$ is another unipotent $k$-group,
then $\omega(\GG\times_k \HH) = \omega(\GG) + \omega(\HH)$.
Let $\GG[\varepsilon]$ denote the $k$-group attached to $\bm\fg[\varepsilon] =
\bm\fg \otimes_k k[\varepsilon]$ from \S\ref{s:examples},
where $\bm\fg$ is the Lie algebra of $\GG$
and $k[\varepsilon]=k[X]/X^2$.

\begin{question}
  \label{qu:3over2}
  Do we always have $\omega(\GG[\varepsilon]) = \frac 3 2 \omega(\GG)$?
\end{question}

For example, the groups with Lie algebras $L_{6,24}(0)$ and
$L_{6,24}(0)[\varepsilon]$ in Tables~\ref{tab:six}--\ref{tab:beyond} have
$\omega$-invariants $\frac 5 2$ and $\frac{15} 4$, respectively. 
Looking at the various examples in Tables~\ref{tab:six}--\ref{tab:beyond}, one
cannot help but notice factors $s^e$ ($e\ge 1$) in the numerators of all functions given there.
\begin{question}
  \label{qu:vanish}
  Does $\zeta_{\GG,\topo}(s)$ always vanish at zero?
\end{question}

Finally, fixed points of topological representation zeta functions seem to exhibit
properties similar to zeros of topological subalgebra zeta functions,
cf.~\cite[Conj.~III]{topzeta}.
\begin{question}
  \label{qu:fp_bounds}
  Let $s_0 \in \CC$ with $\zeta_{\GG,\topo}(s_0)=s_0$.
  Do we have $0\le \Real(s_0) \le \dim(\GG)-1$?
\end{question}

{
  \bibliographystyle{abbrv}
  \tiny
  \bibliography{unipotent}
}

\end{document}